\documentclass{amsart}
\usepackage{amssymb,amsmath,amsthm,graphics,amscd,amsfonts}

\theoremstyle{plain}
\newtheorem{theorem}{Theorem}[section]
\newtheorem{proposition}[theorem]{Proposition}
\newtheorem{corollary}[theorem]{Corollary}
\newtheorem{lemma}[theorem]{Lemma}

\newtheorem{example}[theorem]{Example}
\newtheorem{definition}[theorem]{Definition}

\newcommand{\bfo}{{\bf o}}

\newcommand{\bfC}{{\mathbb C}}
\newcommand{\bfP}{{\mathbb P}}
\newcommand{\bfR}{{\mathbb R}}

\newcommand{\barj}{{\overline j}}

\newcommand{\barl}{{\overline \ell}}

\newcommand{\barz}{{\overline z}}

\newcommand{\barpartial}{{\overline \partial}}

\newcommand{\mapright}[1]{\smash{\mathop{   \hbox to 0.7cm{\rightarrowfill}}
  \limits^{#1}}}

\newcommand{\Ker}{{\rm Ker}}
\newcommand{\Vol}{{\rm Vol}}

\begin{document}
\title
{Einstein metrics and GIT stability}
\author{Akito Futaki}
\address{Department of Mathematics, Tokyo Institute of Technology, 2-12-1
O-okayama, Meguro, Tokyo 152-8551, Japan}
\email{futaki@math.titech.ac.jp}
\author{Hajime Ono}
\address{Department of Mathematics,
Faculty of Science and Technology,
Tokyo University of Science,
2641 Yamazaki, Noda,
Chiba 278-8510, Japan}
\email{ono\_hajime@ma.noda.tus.ac.jp}


\begin{abstract} In this expository article we review the problem of finding Einstein metrics 
on compact K\"ahler manifolds and Sasaki manifolds. In the former half of this article we see that, 
in the K\"ahler case,  the problem fits better with the notion of stability in Geometric Invariant Theory
if we extend the problem to that of finding extremal K\"ahler metrics or constant
scalar curvature K\"ahler (cscK) metrics. In the latter half of this paper we see that most of ideas in K\"ahler geometry extend
to Sasaki geometry as transverse K\"ahler geometry. We also summarize recent results about the existence of toric  Sasaki-Einstein
metrics. 
\end{abstract}


\maketitle

\section{Introduction}

As the Riemannian metrics of constant curvature on compact Riemann surfaces are used in Teichm\"uller theory
it is an important problem to find a metric which is canonical in a certain sense on a given K\"ahler manifold. A typical such result would be the proof of the Calabi conjecture published in 1977 by Yau (\cite{yau77}). This result says that given a compact K\"ahler manifold $M$ with $c_1(M) = 0$ there exists a unique K\"ahler metric with the Ricci curvature identically zero, called a Ricci-flat K\"ahler metric, in each K\"ahler class, and now a compact K\"ahler manifold with $c_1(M) = 0$ is called a Calabi-Yau manifold. In the case when $c_1(M) < 0$, namely in the case when the first Chern class is represented by a real closed $(1,1)$-form whose coefficients form a negative definite Hermitian matrix, the existence of a K\"ahler metric with the Ricci curvature equal to $-1$ times the K\"ahler metric, called a K\"ahler-Einstein metric with negative sign, was proved by Yau (\cite{yau77}) and also by Aubin (\cite{aubin76}) around the same time as the Calabi conjecture. On the other hand in the case when $c_1(M) > 0$ the problem of finding a K\"ahler-Einstein metric of positive sign is known to have various obstructions, and complete understanding has not been obtained. A compact K\"ahler manifold with $c_1(M) > 0$ is called a Fano manifold. By a theorem of Matsushima (\cite{matsushima57}) the complex  Lie algebra $\mathfrak h(M)$ of all holomorphic vector fields on a compact K\"ahler-Einstein manifold is reductive, and by a result  of the first author (\cite{futaki83.1}) for any given Fano manifold $M$ there exists a Lie algebra character $f : \mathfrak h(M) \to \bfC$ with the property that $f = 0$ if $M$ admits a K\"ahler-Einstein metric. These two results therefore give obstructions to the existence of K\"ahler-Einstein metrics. On the other hand Yau conjectured that the existence of K\"ahler-Einstein metrics in the case of $c_1(M) > 0$ will be equivalent to certain sense of stability in Geometric Invariant Theory (GIT for short) (\cite{yau93}). This conjecture comes from the well-known results about finding Hermitian-Einstein metrics on holomorphic vector bundles over compact K\"ahler manifolds. If one can find an Hermitian-Einstein metric on a holomorphic vector bundle over a compact K\"ahler surface, it gives a special case of anti-self dual connection. In the gauge theoretic study of four manifolds in 1980's Donaldson proved that the existence of an Hermitian-Einstein metric on a holomorphic vector bundle over a compact K\"ahler surface is equivalent to the stability of the vector bundle in the sense of Mumford and Takemoto, a kind of GIT stability. In the case of holomorphic vector bundles over compact K\"ahler manifolds of higher dimensions a similar result was proved by Uhlenbeck and Yau, see  \cite{donkro} for the detail of these results. Yau's conjecture for K\"ahler-Einstein metrics in the positive case suggests that the GIT stability should play the key role similarly to the vector bundle case.

The necessity of GIT stability was shown first by Tian \cite{tian97}. He introduced the notion of K-stability and proved that if a Fano manifold $M$ admits a K\"ahler-Einstein metric then $M$ is K-stable. To define K-stability one first considers degenerations of Fano manifolds as algebraic varieties then defines K-stability using $f$ as a numerical invariant to measure the stability. Tian also proved that the ``properness'' of Mabuchi K-energy is equivalent to the existence of K\"ahler-Einstein metric. Mabuchi K-energy amounts to the log of Quillen metric of the determinant line bundle of certain elliptic operator over the space of K\"ahler metrics (see Tian \cite{tian94} ). The space of K\"ahler metrics can be regarded as an orbit of the action of symplectic diffeomorphisms, and the properness of Mabuchi K-energy and the stability are therefore considered to be equivalent. Thus the properness of Mabuchi K-energy implies GIT stability. Such an explanation has been given in the case of Hermitian-Einstein vector bundles (\cite{donkro}). 
Fujiki \cite{fujiki92} and Donaldson \cite{donaldson97}  used the moment map picture of GIT stability to show the relationship between the existence of K\"ahler metric of constant scalar curvature and GIT stability.

The facts mentioned above can be found in earlier Sugaku articles by the first author \cite{futaki92} in 1992 and Bando  \cite{bando98} in 1998. The book written by Nakajima  \cite{nakajima98} in 1999 also include the detail of the above facts. The present article therefore is restricted only to the later development. However, after the two papers \cite{donaldson01} in 2001 and  \cite{donaldson02} in 2002 by Donaldson, papers in this field increased rapidly and it is not possible to cover all of them. We restrict ourselves therefore mainly to our own results and to own interest, and omit many important results by other authors.

Below is the summary of later sections. In section 2 we see the moment map picture of GIT stability and show that the scalar curvature becomes the moment map following the arguments of Fujiki and Donaldson. This shows that K\"ahler metrics of constant scalar curvature or extremal K\"ahler metrics are more directly related to GIT stability rather than K\"ahler-Einstein metrics. Though Matsushima's theorem and the character $f$ were obtained first as obstructions to the existence of K\"ahler-Einstein metrics they are extended to obstructions to the existence of K\"ahler metrics of constant scalar curvature metrics. Matsushima's theorem is further extended as a structure theorem for the Lie algebra of all holomorphic vector fields on compact K\"ahler manifolds with extremal K\"ahler metrics (\cite{calabi85}). We see that these results can be obtained by applying the proofs for the corresponding facts in the framework of moment map picture of the finite dimensional model. The Lie algebra character $f$ can be extended to obstructions for higher Chern forms to be harmonic (\cite{bando83}). Using this fact we consider a perturbation of extremal K\"ahler metrics by perturbing the scalar curvature incorporating the higher Chern forms. In this case again the finite dimensional model suggests the right proofs of the results which are expected to be true (\cite{futaki06}, \cite{futaki07.1}).

In section 3 we discuss the relationship between the existence of K\"ahler metrics of constant scalar curvature and asymptotic Chow semistability.  We see first of all that the character $f$ appears as an obstruction to asymptotic Chow semistability. We then discuss on the results obtained by Donaldson (\cite{donaldson01}) and others.

In section 4 we state the conjecture about the equivalence of K-stability and the existence of K\"ahler metrics of constant scalar curvature. Donaldson re-defined the character $f$ in the manner of algebraic geometry, and used it to re-define the notion of K-stability. Under the new definitions it is shown that the existence of K\"ahler metrics of constant scalar curvature implies K-semistability (\cite{chentian04},                   \cite{donaldson05}). In Donaldson's definition of K-stability one uses the pair of a K\"ahler manifold and an ample line bundle $L$, but for general K\"ahler manifold with general K\"ahler class one can define K-stability using the geodesics in the space of K\"ahler metrics and the behavior of Mabuchi K-energy along the geodesics.  It is conjectured that K-stability is a necessary and sufficient condition for the existence of K\"ahler metrics of constant scalar curvature.

In section 5 we will discuss on the existence problem of Einstein metrics on Sasaki manifolds, called Sasaki-Einstein metrics. Sasaki-Einstein metrics are studied in recent years both in mathematics and physics since they play an important role in the study of AdS/CFT correspondence in superstring theory. Sasaki manifolds are contact Riemannian manifolds whose cone is a K\"ahler manifold, and thus they are odd dimensional. The Reeb vector field defined by the contact structure admits a transverse K\"ahler structure. If a Sasaki manifold admits an Einstein metric then the Ricci curvature is necessarily positive, and thus if the manifold is complete then it is compact. Further the transverse K\"ahler structure also admits positive K\"ahler-Einstein metric. If $M$ is a Fano manifold and $S$ is the total space of the $U(1)$-bundle associated with the canonical line bundle then $S$ is a Sasaki manifold. If $M$ admits no nontrivial holomorphic vector field then finding a Sasaki-Einstein metric on $S$ is equivalent to finding a K\"ahler-Einstein metric. Thus it is apparent that the stability plays a role in this case. However if $M$ admits a nontrivial torus action then the Sasaki structure on $S$ can be deformed by the deformations of Reeb vector field inside the Lie algebra of the torus, and hence there is more possibility for $S$ to admit a Sasaki-Einstein metric. In fact the authors were able to prove that a $(2m+1)$-dimensional Sasaki manifold admitting an effective action of $(m+1)$-dimensional torus admits a Sasaki-Einstein metric if $S$ is described by a ``toric diagram of height $\ell$''
(\cite{FOW}, \cite{CFO}). In particular the total space of $U(1)$-bundle associated with the canonical line bundle $K_M$ of a toric Fano manifold $M$ admits a Sasaki-Einstein metric. Applying this we can show that for any positive integer $k$, the $k$-fold connected sum $k(S^2 \times S^3)$ of $S^2 \times S^3$ carries countably many deformation inequivalent toric Sasaki-Einstein metrics (\cite{CFO}). As another application we can prove the existence of a complete Ricci-flat K\"ahler metric on the total space of the canonical line bundle $K_M$ of a toric Fano manifold $M$ (\cite{futaki07.2}).

\section{Symplectic geometry and scalar curvature}

Let $(Z, \Omega)$ be a K\"ahler manifold and suppose that a compact Lie group $K$ acts on $Z$ as holomorphic isometries. Then the complexification $K^c$ of $K$ acts on $Z$ as biholomorphisms. The actions of $K$ and $K^c$ induce homomorphisms of the Lie algebras ${\mathfrak k}$ and ${\mathfrak k}^c$ of $K$ and $K^c$ to the real Lie algebra $\Gamma(TZ)$ of all smooth vector fields on $Z$. We shall denote these homomorphisms by the same letter $\rho$. Then for $\xi,\ \eta \in {\mathfrak k}$, $\xi + i\eta \in {\mathfrak k}^c$ we have 
$$ \rho(\xi + i\eta) = \rho(\xi) + J\rho(\eta),$$
where $J$ denotes the complex structure of $Z$. Let $[\Omega]$ be a de Rham class which represents an integral cohomology class, and let $ L \to Z$ be the holomorphic line bundle with $c_1(L) = [\Omega]$. There exists an Hermitian metric $h$ on $L^{-1}$ such that its Hermitian connection $\theta$ satisfies
$$ - \frac 1{2\pi} d\theta = \pi^{\ast}\Omega$$
where $\pi : L^{-1} \to Z$ denotes the projection. If we lift the action of $K^c$ to $L^{-1}$ then a moment map $\mu : Z \to {\mathfrak k}^{\ast}$ is determined (see \cite{donkro}, section 6.5). Suppose that $p \in L^{-1} - \mathrm{zero\ section}$ and $\ x \in Z$ satisfy $\pi (p) = x$. Let $\Gamma = K^c\cdot x$ be the $K^c$-orbit through $x \in Z$, and $\widetilde{\Gamma} = K^c\cdot p$ be the $K^c$-orbit through $p \in L^{-1}$. We say that $x \in Z$ is {\bf polystable} with respect to the $K^c$-action if $\widetilde{\Gamma}$ is a closed subset in the total space of $L^{-1}$. We define a function $\ell : \widetilde{\Gamma} \to {\mathbb R}$ on $\widetilde{\Gamma}$ by
$$ \ell(\gamma) = \log |\gamma|^2$$
where the norm $|\gamma|$ is taken with respect to $h$. The following is well-known, see \cite{donkro}, section 6.5 again. 
\begin{itemize}
\item\ \ The function $\ell$ has a critical point if and only if the moment map $\mu : Z \to \mathfrak k^{\ast}$ has a zero on $\Gamma$.
\item\ \ The function $\ell$ is convex.
\end{itemize}
From these two facts the next propositions follow.

\begin{proposition} 
The point $x \in Z$ is polystable with respect to the action of $K^c$ if and only if the moment map $\mu$ has a zero on $\Gamma$.
\end{proposition}

\begin{proposition}\label{W0}
There is at most one connected component of the zero set $\{x \in \Gamma\ |\ \mu(x) = 0\}$ of the moment map on $ \Gamma$.
Further if $\{x \in \Gamma\ |\ \mu(x) = 0\}$ is not empty the function $\ell$ takes its minimum on 
$\{p \in \widetilde{\Gamma}\ |\ \mu(\pi(p)) = 0\}$ and thus  $\ell$ is bounded from below.
\end{proposition}

Fixing $x \in Z$, we denote by $\mu(x) : {\mathfrak k}^c \to \bfC$ the $\bfC$-linear extension of 
$\mu(x) : {\mathfrak k} \to \bfR$. Let $K_x$ and $(K^c)_x$ be the stabilizer subgroups at $x$ of the action of 
$K$ and $K^c$, and let ${\mathfrak k}_x$ and $({\mathfrak k}^c)_x$ be their respective Lie algebras.
Let $f_x : ({\mathfrak k}^c)_x \to {\mathbb C}$ be the restriction of $\mu(x) : {\mathfrak k}^c \to \bfC$ to $({\mathfrak k}^c)_x$.
Notice that $(K^c)_{gx} = g(K^c)_x g^{-1}$. See  \cite{xwang04} or \cite{futaki05}  for the proofs of Proposition \ref{W1} and Proposition \ref{W2}.

\begin{proposition}[\cite{xwang04}]\label{W1}
Fix $x_0 \in Z$.  Then for $x \in K^c\cdot x_0$, $f_x$ is $K^c$-equivariant, that is $f_{gx}(Y) = f_x(Ad(g^{-1})Y)$.
In particular if $f_x$ vanishes for some $x \in K^c\cdot x_0$ then it vanishes for every $x \in K^c\cdot x_0$.
 Furthermore $f_x : ({\mathfrak k}^c)_x \to {\mathbb C}$ becomes a Lie algebra homomorphism.
 \end{proposition}

Suppose we are given a $K$-invariant inner product on ${\mathfrak k}$. Then we have a natural identification
${\mathfrak k} \cong {\mathfrak k}^{\ast}$, and ${\mathfrak k}^{\ast}$ also has a 
$K$-invariant inner product.
Let us consider the function $\phi : K^c\cdot x_0 \to {\mathbb R}$ given by
$\phi(x) = |\mu(x)|^2$. A critical point $x \in K^c\cdot x_0$ of $\phi$ is called an
 {\bf extremal point}.

\begin{proposition}[\cite{xwang04}]\label{W2}
Let $x \in K^c\cdot x_0$ be an extremal point. Then we have a decomposition of the Lie algebra
$$ ({\mathfrak k}^c)_x = ({\mathfrak k}_x)^c + \sum_{\lambda > 0} {\mathfrak k}^c_{\lambda} $$
where ${\mathfrak k}^c_{\lambda}$ is the $\lambda$-eigenspace of ${\mathrm ad}(\sqrt{-1}\mu(x))$, and
$\sqrt{-1}\mu(x)$ belongs to the center of $({\mathfrak k}_x)^c$. In particular we have
$({\mathfrak k}_x)^c = ({\mathfrak k}^c)_x$ if $\mu(x) = 0$.
\end{proposition}

We wish to extend the above results to K\"ahler geometry. For this purpose let us recall basic definitions in K\"ahler geometry.
A K\"ahler metric $g = (g_{i{\barj}})$ on a compact K\"ahler manifold $M$ is called an {\bf extremal K\"ahler metric} if  the  $(1,0)$-part 
$$\mathrm{grad}^{1,0}S = \sum_{i,j = 1}^m 
g^{i\barj}\frac{\partial S}{\partial z^{\barj}}
\frac{\partial}{\partial z^i}$$
of the gradient vector field of the scalar curvature $S$ is a holomorphic vector field. An
extremal K\"ahler metric is a critical point of the functional
$$ g \mapsto \int_M |S|^2 dV_g $$
on the space of all K\"ahler metrics in a fixed K\"ahler class.
If the scalar curvature $S$ is constant then its gradient vector field is zero, and in particular a holomorphic vector field and thus the metric is an
extremal K\"ahler metric. A K\"ahler-Einstein metric is a K\"ahler metric whose Ricci curvature
$$ R_{i{\barj}} = - \frac{\partial^2}{\partial z^i \partial {\overline z}^j} \log \det g$$
is proportional to the K\"ahler metric $g$. Then there exists a real constant $k$ such that
\begin{equation}
 R_{i{\barj}} = k g_{i\barj}.
\end{equation}
Such a metric has constant scalar curvature and a K\"ahler-Einstein metric is an
extremal K\"ahler metric. On the other hand the Ricci form 
$$\rho_g = \frac {\sqrt{-1}}{2\pi} \sum_{i,j=1}^{m} R_{i{\barj}} dz^i \wedge d{\bar z}^j $$
represents the first Chern class $c_1(M)$ as a de Rham class. 
In accordance with the sign of $k$,
$c_1(M)$ is represented by a positive, $0$ or negative
$(1,1)$-form. We express these three cases by writing
$c_1(M) > 0$, $c_1(M) = 0$ or $c_1(M) < 0$.
Apparently it is necessary for $M$ to admits a K\"ahler-Einstein metric that one of the three conditions is satisfied.
One may ask the converse. The cases when $c_1(M) < 0$ and $c_1(M) = 0$ has been settled while the case when $c_1(M) > 0$
has not been completely settled as was explained in section 1. 

In the usual arguments in K\"ahler geometry the complex structure is fixed and some K\"ahler class $[\omega_0]$ of a K\"ahler form is fixed, and then
consider the variational problem of finding extremal K\"ahler metrics by varying the K\"ahler form $\omega$ in the de Rham class $[\omega_0]$.
On the other hand we will consider later the moment map given by the scalar curvature where 
$\omega$ is fixed and $\omega$-compatible complex structure $J$ is varied.
As a matter of fact a variational problem in this setting leads to extremal K\"ahler metrics as critical points.
Later we will study perturbed scalar curvature and see that the perturbed extremal K\"ahler metrics are obtained as a critical point of the variations of $\omega$-compatible
complex structures but not obtained as a critical point of the variations of K\"ahler forms compatible with fixed complex structure $J$ (\cite{futaki06}, \cite{futaki07.1}).

Let  ${\mathfrak h}(M)$ denote the complex Lie algebra of all holomorphic vector fields on 
$M$ and set
$${\mathfrak h}_0(M) = \{ X \in {\mathfrak h}(M)\, |\,X\ \text{has a zero} \}.$$
It is a well-known result (\cite{Lic} or \cite{lebrunsimanca93}) that for $X \in {\mathfrak h}_0(M)$ there exists a unique complex-valued smooth function 
$u_X$ such that
\begin{equation}\label{eq2}
 i(X) \omega = - \barpartial u_X.
  \end{equation}

In this sense ${\mathfrak h}_0(M)$ coincides with the set of all
``Hamiltonian'' holomorphic vector fields. 
(The terminology ``Hamiltonian'' may be misleading because $X$ does not preserve the symplectic form unless $u_X$ is a real valued function. )
We always assume that 
Hamiltonian function $u_X$ is normalized as
\begin{equation}\label{eq3}
 \int_M u_X\,\omega^m = 0.
\end{equation}

Let $(M, \omega_0, J_0)$ be a compact K\"ahler manifold where $\omega_0$ denotes a K\"ahler 
form and $J_0$ a complex structure. We assume $\dim_{\bfR}M = 2m$. In what follows $\omega_0$ shall be
a fixed symplectic form and the complex structures shall be varied. Let 
$Z$ be the set of all complex structures $J$ which are compatible with $\omega_0$. Here, we say that $J$ is compatible with $\omega_0$ if
$$ \omega_0(JX, JY) = \omega_0(X,Y), \quad \omega_0(X,JX) > 0$$
are satisfied for all $X,\ Y \in T_pM$. Therefore, for each $J \in Z$, 
the triple $(M, \omega_0, J)$ is a K\"ahler manifold. In this situation the tangent space of $Z$ at $J$ is a subspace of the space  $\mathrm{Sym}^2(T^{\ast 0,1}M)$ 
of symmetric tensors of type $(0,2)$, and the natural $L^2$-inner product on $\mathrm{Sym}^2(T^{\ast 0,1}M)$ gives 
$Z$ a K\"ahler structure. 

The set of all smooth functions $u$ on $M$ with
$$ \int_M u\, \omega_0^m/m! = 0 $$
is a Lie algebra with respect to the Poisson bracket in terms of $\omega_0$. 
Denote this Lie algebra by ${\mathfrak k}$ and let  $K$  be its Lie group. Namely $K$ is a subgroup of the group
of symplectomorphisms generated by Hamiltonian diffeomorphisms.  $K$ acts on the K\"ahler manifold $Z$ as
holomorphic isometries.

\begin{theorem}[\cite{fujiki92}, \cite{donaldson97}]\label{FujDon} Let $S_J$ be the scalar curvature of the K\"ahler manifold $(M, \omega_0, J)$
and let $\mu : Z \to {\mathfrak k}^{\ast}$ be the map given by
$$ <\mu(J), u> = \int_M S_J\, u\, \omega_0^m $$
where $ u \in \mathfrak k$. Then $\mu$ is a moment map for the action of  $K$ on $Z$.
\end{theorem}

In this situation there is no action on  $Z$ of the complexification  $K^c$ of $K$. However there is a natural infinitesimal action on $Z$ of the complexified Lie algebra $\mathfrak k^{\bfC}$. 
This gives $Z$ a foliation structure and each leaf can be regarded as the set of  all $(\omega_0, J)$ which corresponds to $(\omega, J_0)$ with 
$[\omega] = [\omega_0]$ via Moser's theorem.
In this sense each leaf can be regarded as a space of K\"ahler forms in a given K\"ahler class. 
Ignoring this subtlety one may apply Propositions \ref{W0}, 
\ref{W1} and \ref{W2} to $Z$ formally then they imply three well-known results in K\"ahler geometry which we now explain.

Before explaining them let us digress by a remark. Theorem \ref{FujDon} implies that $J$ is a critical point of
$$ J \mapsto \int_M |S_J|^2 \omega_0^m $$
if and only if $(M, J, \omega_0)$ is an extremal K\"ahler manifold. 

First of all Proposition \ref{W0} implies the following result. 
The function $\ell$ in Proposition \ref{W0} amounts to a functional on the space of K\"ahler forms in a given K\"ahler class, called the Mabuchi K-energy.

\begin{theorem}[\cite{chentian04}]
Let $M$ be a compact K\"ahler manifold, $[\omega_0]$ a fixed K\"ahler class. There is at most one connected component of the space of all
constant scalar curvature K\"ahler metrics in $[\omega_0]$.
If there is one component, the Mabuchi K-energy attains its minimum on this component.
In particular the Mabuchi K-energy is bonded from below if there exits a constant scalar curvature K\"ahler metric in $[\omega_0]$.
\end{theorem}

Let us see next what Proposition \ref{W1} implies.
First of all, since $x \in Z$ is an $\omega_0$-compatible complex structure,
its stabilizer subgroup 
$K_x$ consists of all biholomorphisms expressed as Hamiltonian diffeomorphisms.

Choose any $\omega \in [\omega_0]$. Then by (\ref{eq2}), a Hamiltonian holomorphic vector field $X$ 
is expressed as $X = \sqrt{-1} \mathrm{grad}^{1,0} u_X$. Here, $\mathrm{grad}^{1,0} u_X$ 
is the $(1,0)$-part 
$$\mathrm{grad}^{1,0} u_X = \sum_{i,j = 1}^m 
g^{i\barj}\frac{\partial u_X}{\partial z^{\barj}}
\frac{\partial}{\partial z^i}$$
of the gradient vector field of $u_X$. 
Then by Proposition \ref{W1} we obtain a Lie algebra homomorphism 
\begin{equation}\label{W3}
f(X) :=  -\sqrt{-1}\langle \mu(J), u_X \rangle = -\sqrt{-1} \int_M u_X S_J \omega^m 
 =  \int_M XF \ \omega^m
\end{equation}
where $F \in C^{\infty}(M)$ is given by
$$ \Delta F = S_J - \frac{\int_M S_J \omega^m}{\int_M \omega^m}.$$

\begin{theorem}[\cite{futaki83.1}, \cite{calabi85}]
Let $M$ be a compact K\"ahler manifold, $[\omega_0]$ a fixed K\"ahler class.
Then the Lie algebra homomorphism  $f$ given by  (\ref{W3}) 
does not depend on the choice of a K\"ahler form $\omega \in [\omega_0]$.
Further, if there exists a constant scalar curvature K\"ahler metric in the K\"ahler class $[\omega_0]$ then we have $f = 0$.
\end{theorem}

Proposition \ref{W2} implies the following.

\begin{theorem}[\cite{calabi85}]\label{decomp}
Let $M$ be a compact extremal K\"ahler manifold. Then the Lie algebra ${\mathfrak h}(M) $ has a semi-direct sum decomposition
$$ {\mathfrak h}(M) = \mathfrak h_0 + \sum_{\lambda > 0}\mathfrak h_{\lambda} $$
where $\mathfrak h_{\lambda}$ is the $\lambda$-eigenspace of $\mathrm{ad}(\sqrt{-1}\mathrm{grad}^{1,0} S)$,
and
$\sqrt{-1}\mathrm{grad}^{1,0}S$ belongs to the center of $\mathfrak h_0$.
Further $\mathfrak h_0$ is reductive.
\end{theorem}

From this theorem it follows that if $M$ admits a constant scalar curvature K\"ahler metric then we have  ${\mathfrak h}(M) = \mathfrak h_0$, and
therefore ${\mathfrak h}(M) $ is reductive. This result is called the Lichnerowicz-Matsushima theorem and is a well-known obstruction
for the existence of K\"ahler metrics of constant scalar curvature.

Next, we consider the case of perturbed scalar curvature, and see that 
we obtain similar results to the unperturbed case as symplectic geometry and dissimilar results as K\"ahler geometry.

In what follows we use $\omega$ instead of $\omega_0$ to denote a fixed K\"ahler form.
For a pair $(J, t)$ of a real number with sufficiently small $t$
and an $\omega$-compatible complex structure $J \in Z$, we define a smooth function  $S(J,t)$ on 
$M$ by
\begin{equation}\label{S(J,t)}
 S(J,t)\, \omega^m = c_1(J) \wedge \omega^{m-1} + 
t c_2(J) \wedge \omega^{m-2} + \cdots + t^{m-1} c_m(J).
\end{equation}
Here $c_i(J)$ denotes the $i$-th Chern form with respect to $(J, \omega)$, which is defined by
\begin{equation}\label{Chern}
\det(I + \frac{\sqrt{-1}}{2\pi}t\Theta) = 1 + tc_1(J) + \cdots + t^m c_m(J)
\end{equation}
where  $\Theta$ denotes the curvature form of the Levi-Civita connection for $(J, \omega)$. 

We say that the K\"ahler metric $g$ of a K\"ahler manifold
$(M, J, \omega)$ is a {\bf $t$-perturbed extremal K\"ahler metric} or simply perturbed extremal K\"ahler metric
if 
\begin{equation}\label{grad}
\mathrm{grad}^{1,0}S(J,t) = \sum_{i,j = 1}^m g^{i\barj}\frac{\partial S(J,t)}{\partial \barz^j} \frac{\partial}
{\partial z^i}
\end{equation}
is a holomorphic vector field.

\begin{proposition}[\cite{futaki06}]\label{critical} If we define a functional $\Phi$ on
$Z$ by
\begin{equation}\label{Phi}
\Phi(J) = \int_M S(J,t)^2 \omega^m
\end{equation}
then the critical points of $\Phi$ are perturbed extremal K\"ahler metric.
\end{proposition}
The proof of Proposition \ref{critical} follows from the fact that the perturbed scalar curvature becomes the moment map
(cf. Theorem \ref{symp2} below) just as in the unperturbed case.
The perturbed scalar curvature becomes the moment map with respect to the perturbed symplectic structure
on $Z$ described as follows.
The tangent space of 
$Z$ at $J$ is identified with a subspace of  $\mathrm{Sym}(\otimes^2T^{\ast 0,1}M)$. 
When the real number $t$ is small enough, we define the Hermitian structure on $\mathrm{Sym}(\otimes^2T^{\ast 0,1}M)$ by
\begin{equation}\label{inner}
(\nu, \mu)_t = \int_M mc_m(\overline{\nu}_{jk}
\,\mu^i{}_\barl \frac {\sqrt{-1}}{2\pi}\,
dz^k \wedge d\overline{z^{\ell}}, \omega\otimes I + \frac {\sqrt{-1}}{2\pi}\,t\Theta,
\cdots, \ \omega\otimes I + \frac {\sqrt{-1}}{2\pi}\,t\Theta).
\end{equation}
Here 
$\mu$ and $\nu$ are tangent vectors in $T_J Z$, and $c_m$ is regarded as the polarization of
the determinant which is a $GL(m,\bfC)$-invariant polynomial.
That is,
$c_m(A_1, \cdots, A_m)$ is the coefficient of $m!\, t_1 \cdots t_m$ in 
 $\det(t_1A_1 + \cdots +
t_mA_m)$. Further, 
 $I$ denotes the identity matrix, 
$\Theta = \overline\partial (g^{-1}\partial g)$ denotes the curvature form the Levi-Civita connection and
$u_{jk}\mu^i_{\bar l}$  is regarded as an endomorphism of $T^{1,0}_JM$ that sends $\partial/\partial z^j$ to $u_{jk}\mu^i_{\bar l}{\partial/\partial z^i}$.
When $t = 0$, (\ref{inner}) is the usual $L^2$-inner product. The perturbed symplectic form $\Omega_{J,t}$  at 
$J \in Z$ is given by
\begin{eqnarray}\label{symp}
&&\Omega_{J,t}(\nu,\mu) =
\Re (\nu,\sqrt{-1}\mu)_t \\
&&= \Re \int_M mc_m(\overline{\nu}_{jk}
\,\sqrt{-1}\mu^i{}_\barl \frac {\sqrt{-1}}{2\pi}\,
dz^k \wedge d\overline{z^{\ell}}, \omega\otimes I + \frac {\sqrt{-1}}{2\pi}\,t\Theta, \nonumber \\ 
&& \hspace{6.5cm}\cdots, \ \omega\otimes I + \frac {\sqrt{-1}}{2\pi}\,t\Theta) \nonumber
\end{eqnarray}
where $\Re$ stands for the real part.
\begin{theorem}[\cite{futaki06}] If \label{symp2}$\delta J = \mu$ then we have 
\begin{equation}\label{moment1}
\delta \int_M u\ S(J,t)\omega^m =  
\Omega_{J,t}(2\sqrt{-1}\nabla^{\prime\prime}\nabla^{\prime\prime}u,\mu).
\end{equation}
That is, the perturbed scalar curvature  $S(J,t)$ becomes the moment map for the action of the
group of all Hamiltonian diffeomorphisms with respect to the perturbed symplectic form
$\Omega_{J,t}$.
\end{theorem}

Let us now fix $J$ and write $S(\omega, t)$ for the perturbed scalar curvature in terms of  $\omega \in [\omega_0]$
given by the right hand side of (\ref{S(J,t)}). It appears that the critical points of the functional 
$$ \omega \mapsto \int_M |S(\omega, t)|^2 \omega^m $$
on $[\omega_0]$ are not perturbed extremal K\"ahler metrics
(cf. Remark 3.3 in  \cite{futaki06}).
In the perturbed case we also have results corresponding to Proposition 
\ref{W1} and \ref{W2} as in the unperturbed case. In fact, corresponding to Proposition \ref{W1}  we obtain Bando's obstructions (\cite{bando83}) for
higher Chern forms to be harmonic, see \cite{futaki07.1} and section 3 of this article.
Proposition \ref{W2} suggests that we should have a similar decomposition theorem in the perturbed case.
Using the arguments of L.-J. Wang (\cite{Lijing06}) one can give a rigorous proof of the decomposition theorem, see
\cite{futaki07.1}. However it seems hard to give a rigorous proof of the uniqueness theorem that 
Proposition \ref{W0} suggests.

\section{Asymptotic Chow semistability and integral invariants}

In the previous section we saw how compact K\"ahler manifolds with constant scalar curvature 
can be seen from the viewpoints of GIT stability through the picture of moment maps.
We also saw how the well-known obstructions such as the Lie algebra character $f$ given by  (\ref{W3}) 
and the Lichnerowicz-Matsushima theorem appear in this moment map picture.
In this section we shall see that the character $f$ is an obstruction for the asymptotic Chow semistability.
This fact shows that the existence of constant scalar curvature K\"ahler metric really concerns stability
in algebraic geometry. This section is based on \cite{futaki04-1}.

Let $P_G \to M$ be a holomorphic principal $G$ bundle over a compact K\"ahler manifold.
We assume that $G$ is a complex Lie group acting on $P_G$ as a structure group from the right. 
We further assume that a complex Lie group $H$ acts holomorphically on $P_G$ from the left commuting with right action of $G$.
Therefore, in this case, $H$ acts also on $M$ as automorphisms.

Suppose that the principal $G$-bundle $P_G$ has a connection whose connection form is type $(1,0)$ form on $P_G$.
We call such a connection a type $(1,0)$-connection.
A typical such connection is the canonical connection of the holomorphic frame bundle of an Hermitian holomorphic vector bundle.
Namely there is a unique connection on an Hermitian holomorphic vector bundle such that the connection is compatible with the metric
and $(0,1)$-part of the covariant exterior differentiation is equal to  $\overline{\partial}$. Therefore in this case the connection form is of type $(1,0)$.

Let $\theta$ be a type $(1,0)$-connection form and $\Theta$ be the curvature form.
An element $X$ of the Lie algebra ${\mathfrak h}$ of 
$H$ defines a  $G$-invariant vector field on $P_G$. By the abuse of notation, we write $X$ such a vector field on $P_G$.
Let $I^p(G)$ be the set of all $G$-invariant polynomials of degree $p$ on 
 ${\mathfrak g}$.
 Let $p \ge m$ and, for any $\phi \in I^p(G)$ we define $f_{\phi}$ by
$$ f_{\phi}(X) = \int_M \phi(\theta(X) + \Theta).$$
Then one can prove that 
 $f_{\phi}$ is independent of the choice of the $(1,0)$-connection $\theta$, see \cite{futaki88} for the detail.
 From this it follows that 
 $f_{\phi}$ defines an element of $I^{p-m}(H)$ and can be interpreted as the image under the Gysin map of an element of the equivariant
 cohomology (\cite{futakimorita85}).
As a special case, consider the case when 
$G = GL(m,\bfC)$ and $P_G$ is the frame bundle of the holomorphic tangent bundle of a compact K\"ahler manifold
$M$. Take 
$H$ to be a complex subgroup of the automorphism group of $M$.
$H$ naturally acts on $P_G$.
In this case $I^{\ast}(G)$ is the algebras generated by the elementary symmetric functions of the eigenvalues.

For each $\phi \in I^k(G)$ and $X \in \mathfrak h_0$ we put
\begin{eqnarray}\label{family}
{\mathcal F}_{\phi}(X) &=& (m-k+1) \int_M \phi(\Theta) \wedge u_X\,\omega^{m-k}
\nonumber
\\ & & + \int_M \phi(\theta(X) + \Theta) \wedge \omega^{m-k+1}.
\end{eqnarray}
where $u_X$ is assumed to satisfy the normalization (\ref{eq3}). 

\begin{theorem}[\cite{futaki04-1}]\ \ ${\mathcal F}_{\phi}(X)$ is independent of the choice of the K\"ahler form
$\omega \in [\omega_0]$ on $M$ and also of the choice of the type $(1,0)$-connection form $\theta$ on $P_G$.
\end{theorem}

This family of integral invariants contains as a subfamily the obstructions for higher Chern forms to be harmonic
obtained by Bando \cite{bando83}.
Below is the detail about this.
Let $M$ be a compact K\"ahler manifold and 
$[\omega_0]$ be any K\"ahler class. For any 
K\"ahler form $\omega \in [\omega_0]$,
let $c_k(\omega)$ be its $k$-th Chern form and 
$Hc_k(\omega)$ be the harmonic part of $c_k(\omega)$.
Then there is a $(k-1, k-1)$-form $F_k$ such that
$$ c_k(\omega) - Hc_k(\omega) = \frac {\sqrt{-1}}{2\pi}\partial\barpartial F_k.$$
We define 
 $f_k : {\mathfrak h}(M) \to \bfC$ by
$$ f_k(X) = \int_M L_XF_k \wedge \omega^{m-k+1}.$$
Then one can show that $f_k$ is independent of the choice of $\omega \in [\omega_0]$ and therefore
$f_k$ becomes a Lie algebra homomorphism. If $c_k(\omega)$ becomes a harmonic form for some 
$\omega \in [\omega_0]$ then $f_k = 0$ for such an $\omega$. Hence we have $f_k = 0$ in such a case.
That is to say, $f_k$ is an obstruction for a K\"ahler class $[\omega_0]$ to
admit a K\"ahler form such that its $k$-th Chern form is a harmonic form.

In the case when $k= 1$, the first Chern form being harmonic is equivalent to the scalar curvature being constant.
This fact can be checked easily using the second Bianchi identity.
Thus $f_1$ is an obstruction for the K\"ahler class $[\omega_0]$ to admit a constant scalar curvature K\"ahler (cscK) metric.
In fact $f_1$ coincides with $f$ which we defined by  (\ref{W3}) in the last section.

Let us see that 
$${\mathcal F}_{c_k}(X) = (m-k+1) f_k(X)$$
when $P_G$ is the frame bundle of the holomorphic tangent bundle of  $M$ and $\theta$ is the Levi-Civita connection of
the K\"ahler form $\omega$.
As we will see below the second term of (\ref{family}) is  $0$ for $\phi = c_k$.
Next, $Hc_k(\omega) \wedge \omega^{m-k}$ is harmonic, and by the uniqueness of the harmonic form in each cohomology class
this must be a multiple of the volume form $\omega^m/m!$. Then by the normalization condition (\ref{eq3}), 
${\mathcal F}_{c_k}(X)$  coincides with $ (m-k+1)f_k(X)$.

One can prove that, for
$\phi = c_k$, the second term in (\ref{family}) is $0$ in the following way.
$\theta(X)$ is conjugate with $L(X) = L_X - \nabla_X$, but in the K\"ahler case the latter is equal to $\nabla X = 
\nabla \mathrm{grad}^{1,0}u$. Moreover in the calculation of 
\begin{eqnarray*}
{}&&\int_M c_p(\theta(X), \overbrace{\Theta, \cdots, \Theta}^{p-1}) \wedge \omega^{m-p+1}  \\
&& = \int_M c_m(\overbrace{\omega \otimes I, \cdots, \omega \otimes I}^{m-p}, \omega\otimes L(X), 
\overbrace{\Theta,
\cdots, \Theta}^{p-1})
\end{eqnarray*}
we take the determinant both in fiber coordinates and in the base coordinates. Because of this symmetry we have
\begin{eqnarray*}
\text{RHS} &=& \int_M c_m(\omega\otimes I, \cdots, \omega \otimes I, i\partial \barpartial u 
\otimes I,
\Theta, \cdots, \Theta) \\
&=& - \int_M \barpartial c_m(\omega\otimes I, \cdots, \omega \otimes I, i\partial u \otimes I,
\Theta, \cdots, \Theta)\\
&=& 0.
\end{eqnarray*}

Let us next see that, for $1 \le \ell \le m$, ${\mathcal F}_{Td^{\ell}}$'s are obstructions for asymptotic Chow semistability.
Here $Td^{\ell}$ is the $\ell$-th 
Todd polynomial. Geometric invariant theory says that to construct a moduli space with good properties such as Hausdorff property
or quasi-projectivity one has to discard unstable ones (\cite{mumford}). Let
$V$ be a vector space over $\bfC$, and $G$ be a subgroup of $SL(V)$.
We say that $x \in V$ is {\bf stable} if the orbit $Gx$ is closed and if the stabilizer subgroup at $x$ is finite.
We say that $x  \in V$ is {\bf semistable} if the closure of the orbit $Gx$ does not contain the origin 
$\bfo$.

Let $L \to M$ be an ample line bundle.
Put $V_k := H^0(M,L^k)^*$ and let 
$\Phi_{|L^k|} : M \to \bfP(V_k)$ be the Kodaira embedding determined by $L^k$.
Let  $d$ be the degree of $M$ in $\bfP(V_k)$.
A point in the product $\bfP(V_k^*) \times \cdots \times \bfP(V_k^*)$ of  $m+1$ 
copies of $\bfP(V_k^*)$ determines $m+1$ hyperplanes $H_1,\,\cdots\, ,H_{m+1}$
in $\bfP(V_k)$. The set of all $m+1$ hyperplanes $H_1, \,\cdots\, ,H_{m+1}$ such that 
$H_1 \cap \cdots \cap H_{m+1} \cap M$ is not empty defines a divisor in
$\bfP(V_k^*) \times \cdots \times \bfP(V_k^*)$. But since the degree of $M$ is $d$,
this divisor is defined by 
${\hat M}_k \in (Sym^d(V_k))^{\otimes m+1}$. Of course ${\hat M}_k$ is defined up to
a constant. The point 
$[{\hat M}_k] \in \bfP( (Sym^d(V_k))^{\otimes m+1})$ is called the Chow point of $(M, L^k)$.
$M$ is said to be Chow stable with respect to $L^k$ if ${\hat M}_k$ is stable under 
the action of $SL(V_k)$ on $(Sym^d(V_k))^{\otimes m+1}$. $M$ is said to be asymptotically
Chow stable with respect to $L$ if there exists a $k_0 > 0$ such that 
${\hat M}_k$ is stable for all $k \ge k_0$. Asymptotic Chow semistability is defined
similarly.
The stabilizer ${\hat G}_k \subset SL(V_k)$ of ${\hat M}_k$ is a finite covering of a
subgroup $G_k$ of the automorphism group $\mathrm{Aut}(M)$ of $M$. 
If we denote by $\mathrm{Aut}(M, L)$ the subgroup of $\mathrm{Aut}(M)$ consisting of
the elements which lift to an action on $L$, then $G_k$ is a subgroup of $\mathrm{Aut}(M, L)$.

\begin{theorem}[\cite{futaki04-1}]\ \ If $(M, L)$ is asymptotically Chow semistable, then
for $1 \le \ell \le m$ we have 
\begin{equation}
{\mathcal F}_{Td^{\ell}}(X) = 0.
\end{equation}
The case $\ell=1$ implies the vanishing of $f_1$.
\end{theorem}

Note in passing that under the assumption that $\mathrm{Aut}(M, L)$ is discrete, 
Donaldson \cite{donaldson01} obtained the following
results. The K\"ahler form of the Fubini-Study metric of $\bfP(V_k)$ is denoted by $\omega_{FS}$.

\medskip

\begin{enumerate}
\item Suppose that $\mathrm{Aut}(M,L)$ is discrete and 
that $M$ is asymptotically Chow stable. If the sequence of K\"ahler forms 
$\omega_k := \frac{2\pi}k \Phi_{|L^k|}^{\ast}(\omega_{FS})$ belonging in
$c_1(L)$ converges in $C^{\infty}$ to $\omega_{\infty}$, 
then $\omega_{\infty}$ has constant scalar curvature.
\item Suppose that $\mathrm{Aut}(M,L)$ is discrete and that $\omega_{\infty} \in 2\pi c_1(L)$ 
has constant scalar curvature. Then $M$ is asymptotically Chow stable with respect to  $L$, 
and  $\omega_k$ converges in $C^{\infty}$ to $\omega_{\infty}$.
\item Suppose that $\mathrm{Aut}(M,L)$ is discrete. Then a K\"ahler metric of constant 
scalar curvature in $2\pi c_1(L)$ is unique.  \end{enumerate}

\medskip

The case where $\mathrm{Aut}(M,L)$ 
is not discrete is treated by T. Mabuchi in \cite{mabuchiICM}.

\section{K-stability}

In \cite{tian97} Tian defined the notion of K-stability for Fano manifolds
and proved that if a Fano manifold carries a K\"ahler-Einstein metric then
$M$ is weakly K-stable. Tian's K-stability considers the degenerations of 
$M$ to
normal varieties and uses a generalized version of the invariant $f_1$ 
defined by
Ding and Tian (\cite{dingtian92}). Note that this generalized invariant is only defined for normal varieties. 

Further Donaldson re-defined in \cite{donaldson02} the invariant $f_1$ 
for general
polarized varieties (or even projective schemes) and also re-defined 
the notion of K-stability for $(M, L)$. The new definition does not require $M$ to be
Fano nor the central fibers of  degenerations to be normal. We now briefly review 
Donaldson's definition of K-stability.

Let 
$\Lambda \to N$ be an ample line bundle over an $n$-dimensional projective 
scheme. We assume that a ${\mathbb C}^*$-action as bundle isomorphisms of 
$\Lambda$ covering the ${\mathbb C}^*$-action on $N$. 

For any positive integer $k$, there is an induced $\bfC^*$ action on
$W_k = H^0(N, \Lambda^k)$. Put
$d_k = \dim W_k$ and let $w_k$ be the weight of $\bfC^*$-action on 
$\wedge^{d_k}W_k$. 
For large $k$, $d_k$ and $w_k$ are polynomials in $k$ of degree $n$  and $n+1$ respectively
by the Riemann-Roch and the equivariant Riemann-Roch theorems. Therefore 
$w_k/kd_k$ is bounded from above as $k$ tends to infinity.
For sufficiently large $k$ we expand
$$ \frac{w_k}{kd_k} = F_0 + F_1k^{-1} + F_2k^{-2} + \cdots. $$

For an ample line bundle $L$ over a projective variety  $M$, a test configuration of
degree $r$ consists of the following.\\
(1)\ \ A family of schemes $\pi : {\mathcal M} \to \bfC$:\\
(2)\ \ $\bfC^*$-action on ${\mathcal M}$ covering the usual $\bfC^*$-action on $\bfC$:\\
(3)\ \ $\bfC^*$-equivariant line bundle $\mathcal{L} \to {\mathcal M}$ such that
\begin{itemize}
\item 
for $t \ne 0$ one has $M_t = \pi^{-1}(t) \cong M$ and
$(M_t, {\mathcal L}|_{M_t}) \cong (M, L^r)$,
\item 
$\chi(M_t, L^r_t) = \sum_{p=0}^n (-1)^p \dim 
H^p(M_t, L_t^r)$ does not depend on $t$, in particular for $r$ sufficiently large
$\dim H^0(M_t, L_t^r) = \dim H^0(M,L^r)$ for all $t \in \bfC$. Here we write $L^r_t$ for ${\mathcal L}|_{M_t}$
though $L$ may not exist for $t=0$.
\end{itemize}

$\bfC^*$-action induces a $\bfC^*$-action on the central fiber
$L_0 \to M_0 = \pi^{-1}(0)$. Moreover if 
$(M,L)$ admits a $\bfC^*$-action, then one obtains a test configuration
by taking the direct product $M \times \bfC$. This is called a product configuration.
A product configuration is called a trivial configuration if the action of $\bfC^*$ on $M$ is trivial.


\begin{definition}\ \ $(M,L)$ is said to be K-semistable (resp. stable) if
the $F_1$ of the central fiber $(M_0, L_0)$ is non-positive (negative)
for all non-trivial test configurations. $(M,L)$ is said to be
K-polystable if it is K-semistable and $F_1 = 0$ only
if the test configuration is product.
\end{definition}

\noindent
{\bf Conjecture}(\cite{donaldson02}) : A K\"ahler metric of constant scalar curvature will exist in the 
K\"ahler class $c_1(L)$  if and only if $(M,L)$ is K-polystable.

\bigskip

The following lemma shows that $F_1$ coincides with a positive multiple of $-f(X)$ if the central fiber is nonsingular.  
Here $X$ denotes the infinitesimal generator of the $\bfC^{ast}$-action.
The lemma and Tian's analysis on the behavior of Mabuchi K-energy motivates the Conjecture above.
Recall that $\Lambda$ was an ample line bundle with $\bfC^{\ast}$-action
over a projective
scheme $N$ and that $F_1$ was defined for $(N, \Lambda)$. Suppose that $N$ is 
nonsingular algebraic variety and take any K\"ahler form $\omega$ in $c_1(\Lambda)$. Denote by
$\rho$ and $\sigma$ the Ricci form and the scalar curvature of $\omega$ respectively.

\begin{lemma}[\cite{donaldson02}]\ \ If  $N$  is a nonsingular projective variety then 
$$F_1 = \frac{-1}{2 vol(N, \omega)}f_1(X)$$
where $X$ is the infinitesimal generator of the $\bfC^{\ast}$-action and $f_1$ is the
integral invariant defined in section 2.
\end{lemma}

\begin{proof}\ \ Let us denote by $n$ the complex dimension of $N$.
Expand $h^0(\Lambda^k)$ and $w(k)$ as
$$ h^0(\Lambda^k) = a_0k^n + a_1k^{n-1}+ \cdots,$$
$$ w(k) = b_0k^{n+1} + b_1k^n + \cdots.$$
Then by the Riemann-Roch and the equivariant Riemann-Roch formulae 
$$ a_0 = \frac 1{n!}\int_N c_1(\Lambda)^n = vol(N), $$
$$ a_1 = \frac 1{2(n-1)!} \int_N \rho \wedge c_1(\Lambda)^{n-1} = \frac 1{2n!} \int_N \sigma \omega^n, $$
$$ b_0 = \frac 1{(n+1)!}\int_N (n+1) u_X \omega^n, $$
$$ b_1 = \frac 1{n!} \int_N n u_X \omega^{n-1} \wedge \frac 12 c_1(N) + \frac 1{n!} \int_N 
\operatorname{div}X\ \omega^n .$$
The last term of the previous integral is zero because of the divergence formula. Thus
$$ \frac {w(k)}{kh^0(k)} = \frac{b_0}{a_0}(1 + (\frac{b_1}{b_0} - \frac{a_1}{a_0})k^{-1} + \cdots ) $$
from which we have
\begin{eqnarray*}
 F_1 &=& 
 \frac{b_0}{a_0}(\frac{b_1}{b_0} - \frac{a_1}{a_0}) = \frac 1{a_0^2}(a_0b_1 - a_1b_0)\\
 &=& \frac 1{2vol(N)}\int_N u_X(\sigma - \frac 1{vol(N)}\int_N \sigma \frac{\omega^n}{n!})\frac{\omega^n}{n!}\\
 &=&  \frac 1{2vol(N)}\int_N u_X \Delta F \frac{\omega^n}{n!}
 \ =\   \frac{ -1}{2vol(N)}\int_N XF \frac{\omega^n}{n!}\\
 &=&  \frac{ -1}{2vol(N)}f_1(X).
\end{eqnarray*}
\end{proof}

\section{Existence and uniqueness of Sasaki-Einstein metrics}

The notion of Sasaki manifold was introduced by Sasaki and Hatakeyama
in \cite{sasaki}. Sasaki manifolds are a kind of odd dimensional analog of 
K\"ahler manifolds. Because of the similarity to
K\"ahler manifolds,
Sasaki manifolds had not been in the spotlight for a long time.
Especially, it had been believed that the existence problem of
Sasaki-Einstein metrics is reduced to that of Ricci positive K\"ahler-Einstein
metrics on Fano orbifolds.

However the situation changed
drastically
in the late 1990's. It was pointed
out by 
physicists
that Sasaki-Einstein manifolds play an important role in AdS/CFT
correspondence. 
In fact 
Gauntlett, Martelli, Sparks and Waldram
gave infinitely many examples of Sasaki-Einstein manifolds which
are not obtained as the total spaces of $S^1$-orbibundles of locally cyclic
K\"ahler-Einstein orbifolds, \cite{gmsw}.
We now recognize that
the set of Sasaki-Einstein manifolds is strictly larger than
that of locally cyclic Ricci positive K\"ahler-Einstein orbifolds.
In fact, in the joint works of the authors, Guofang Wang and Koji Cho,
\cite{FOW}, \cite{CFO}, we solve the existence and uniqueness problem
of toric Sasaki-Einstein manifolds completely. Then we see that
there are
much
more toric Sasaki-Einstein manifolds than
toric K\"ahler-Einstein orbifolds.
In the present chapter, we
will
explain such existence and uniqueness
results of Sasaki-Einstein manifolds.

\subsection{Sasaki manifolds}

First of all, we define Sasaki manifolds
as follows.
Let $(S,g)$ be a Riemannian manifold. We denote its Riemannian cone
$(\mathbb{R}_+\times S,dr^2+r^2g)$ by $(C(S),\bar{g})$.

\begin{definition}
A Riemannian manifold $(S,g)$ is said to be a {\bf Sasaki manifold} if 
the Riemannian cone $(C(S),\overline{g})$ is K\"ahler.
\end{definition}
The dimension of Sasaki manifold $(S,g)$ is
odd, and
$(S,g)$ is isometric to the submanifold $\{r = 1\} = \{1\} \times S 
\subset (C(S),\overline{g})$.
When a Sasaki manifold $(S,g)$ is given, there are some important objects
associated with it; let $J$ be a complex structure on 
$C(S)$ such that
$(C(S),\,J,\,\bar{g})$ 
is K\"ahler.
Then we get the vector field $\Tilde{\xi}$ and the $1$-form $\Tilde{\eta}$
on $C(S)$ defined as
$$\Tilde{\xi}=Jr\frac{\partial }{\partial r},\ \ 
\Tilde{\eta}=\frac{1}{r^2}\bar{g}(\Tilde{\xi},\cdot)
=\sqrt{-1}(\bar{\partial}-\partial)\log r.$$
It is easily seen that the restrictions $\xi=\Tilde{\xi}_{|S}$ and
$\eta=\Tilde{\eta}_{|S}$ to $\{r=1\}\simeq S$ give a vector field and a
$1$-form on $S$. These are usually called the Reeb vector field
and the contact form respectively in Contact geometry context.
By the abuse of terminology we call $\Tilde{\xi}$ and $\Tilde{\eta}$ 
the {\bf Reeb vector field} and
the {\bf contact form} on $C(S)$ respectively.
The Reeb vector field $\Tilde{\xi}$ is a Killing vector field on
$(C(S),\bar{g})$ with the length $\bar{g}(\Tilde{\xi}, \Tilde{\xi})^{1/2}=r$. 
The complexification $\Tilde{\xi}-\sqrt{-1}J\Tilde{\xi}$
of the Reeb vector field is holomorphic on $(C(S),J)$.

The K\"ahler form $\omega$ of
$(C(S),\,J,\,\bar{g})$ 
is
$$
\omega=\frac12 d(r^2\Tilde{\eta})=\frac{\sqrt{-1}}{2}\partial
\bar{\partial}r^2.$$

\begin{example}
A typical example of Sasaki manifold is the odd dimensional unit sphere
$S^{2m+1}(1)$. The Riemannian cone of $S^{2m+1}(1)$ is 
$(\mathbb{C}^{m+1}\setminus\{0\},\langle\ ,\ \rangle)$, where $\langle\ ,
\ \rangle$ is the standard inner product.
The Reeb vector field is given by
$$\Tilde{\xi}_0=\sum_{j=0}^m(x^j\frac{\partial }{\partial y^j}-
y^j\frac{\partial }{\partial x^j})
=\sqrt{-1}\sum_{j=0}^m (z^j\frac{\partial}{\partial z^j}
-\bar{z}^j\frac{\partial}{\partial\bar{z}^j}).$$
\end{example}

Now it is obvious that there is a close
relationship
between $(2m+1)$-dimensional
Sasaki manifolds and the Riemannian cones,
which are K\"ahler manifolds
of
complex dimension $m+1$. It is also important to notice that
there is a complex $m$ dimensional K\"ahler structure on the
transverse direction
of the foliation defined by the vector field $\xi$.
Let $(S,g)$ be a Sasaki manifold and $\Tilde{\xi}$ the Reeb vector field on
$C(S)$. Then $\xi=\Tilde{\xi}_{\{r=1\}}$ is a vector field on $S$ and
$g(\xi,\xi)=1$. Hence $\xi$ defines a
one dimensional foliation 
$\mathcal{F}_\xi$ on $S$. We call $\mathcal{F}_\xi$ the Reeb foliation.
On the other hand the holomorphic vector field $\Tilde{\xi}-\sqrt{-1}J
\Tilde{\xi}$ generates a holomorphic flow on $C(S)$.
The local orbits of this flow defines a transversely holomorphic structure,
we denote it by $\Phi$,
which we denote by $\Phi$, on
the Reeb foliation $\mathcal{F}_\xi$ in the following sense.
There are an open covering $\{U_\alpha\}_{\alpha\in A}$ of $S$ and
submersions $\pi_\alpha:U_\alpha\to V_{\alpha}\subset \mathbb{C}^m$
such that when $U_\alpha\cap U_\beta\not=\emptyset$
$$\pi_\alpha\circ \pi_\beta^{-1}:\pi_\beta(U_\alpha\cap U_\beta)\to
\pi_\alpha(U_\alpha\cap U_\beta)$$
is biholomorphic. On each $V_\alpha$ we can give a K\"ahler structure
as follows. Let $D=\Ker \ \eta\subset TS$. 
There is a canonical isomorphism
$d\pi_\alpha:D_p\to T_{\pi_\alpha(p)}V_\alpha$
for any $p\in U_\alpha$.
Since $\xi$ is a Killing vector field on $(S,g)$, the restriction $g_{|D}$
of the Sasaki metric $g$ to $D$ gives a well-defined Hermitian metric
$g_\alpha^T$ on $V_\alpha$. This Hermitian structure is in fact K\"ahler.
The fundamental $2$-form $\omega_\alpha^T$ of $g_\alpha^T$ is the same as
the restriction of $d\eta/2$ to $U_\alpha$.
Hence
we see that $\pi_\alpha\circ \pi_\beta^{-1}:\pi_\beta(U_\alpha\cap U_\beta)
\to\pi_\alpha(U_\alpha\cap U_\beta)$ gives an isometry of K\"ahler manifolds.
Therefore, the Reeb foliation $\mathcal{F}_\xi$ is a transversely K\"ahler
foliation.

From now on we denote by $(S,g;\xi,\eta,\Phi)$ a Sasaki manifold
when we
need to specify
the Reeb vector field, the contact form and
the transverse holomorphic structure.

\begin{example}

The restriction of the Reeb vector field $\xi_0=\Tilde{\xi}_0\, 
_{|S^{2m+1}(1)}$ to $S^{2m+1}(1)$ generates the $S^1$-action
$(z^0,\cdots,z^m)\mapsto (e^{i\theta}z^0,\cdots,e^{i\theta}z^m)$.
Hence the transverse K\"ahler structure of the Reeb foliation $\mathcal{F}
_{\xi_0}$ is identified with the orbit space of the $S^1$-action, that is 
$(\mathbb{C}P^m,g_{FS})$,
where $g_{FS}$ is the Fubini-Study metric.
In general
when the Reeb vector field generates a (locally) free $S^1$-action,
then we call the Sasaki manifold {\bf (quasi-)regular}.
The transverse K\"ahler structure of the Reeb foliation of 
a (quasi-)regular Sasaki manifold is the K\"ahler manifold (locally cyclic
K\"ahler orbifold) obtained as the quotient space of the $S^1$-action.
Conversely when we have a K\"ahler manifold (locally cyclic K\"ahler
orbifold) $(M,\omega)$ with an
integral K\"ahler class $[\omega]$,
we can construct a (quasi-)regular Sasaki metric on $S(L)$ whose transverse
K\"ahler structure is $(M,\omega)$, where $L$ is the complex line 
(orbi)bundle on $M$
with $c_1(L)=-[\omega]$ and $S(L)$ is the associated $U(1)$-bundle.
See \cite{blair}, \cite{boyergalicki} for the detail.

On the other hand we call a Sasaki manifold {\bf irregular}
if the Reeb foliation has a non-closed leaf.
The transverse K\"ahler structure of an irregular Sasaki manifold
cannot be realized as a K\"ahler orbifold.
\end{example}

The Einstein condition of a Sasaki manifold $(S,g)$ is translated into 
Einstein conditions of the Riemannian cone $(C(S),\bar{g})$ or
the transverse K\"ahler structure as follows.

\begin{proposition}
Let $(S,g)$ be a $(2m+1)$ dimensional Sasaki manifold.
Then the following three conditions are equivalent:

1. $g$ is Einstein. Then $Ric_g=2mg$, where $Ric_g$ is the Ricci curvature of
$g$.

2. The Riemannian cone $(C(S),\bar{g})$ is a Ricci-flat K\"ahler manifold. 

3. The transverse K\"ahler metric $g^T$ satisfies $Ric^T=(2m+2)g^T$,
where $Ric_{g^T}$ is the Ricci curvature of $g^T$.
\end{proposition}

\begin{example}
Using the correspondence in Example $5.3$, 
there is one-to-one correspondence between
quasi-regular Sasaki-Einstein manifolds and
Ricci positive locally cyclic K\"ahler-Einstein orbifolds.
Many examples of quasi-regular Sasaki-Einstein manifolds are obtained by
Boyer, Galicki and their collaborators. 
Their results can be found in
\cite{boyergalicki}.

\end{example}

\subsection{Integral invariants}

In Chapter $2$, we saw that there is an obstruction to the existence of
K\"ahler metric of constant scalar curvature defined by $(4)$.
Especially when the K\"ahler class $[\omega]$ equals the first Chern class
of the manifold $c_1(M)$, it gives an obstruction to the existence of
Ricci positive
K\"ahler-Einstein metric.
On the other hand, by Proposition $5.4$, a $(2m+1)$ dimensional Sasaki manifold
is Einstein if and only if the corresponding transverse K\"ahler metric is
Einstein with the Einstein constant $2m+2$.
In this section we would like to define an integral invariant which is an obstruction
for a transverse holomorphic structure to admit a transverse K\"ahler-Einstein
metric.

Let $(S,g;\xi,\eta,\Phi)$ be a compact Sasaki manifold. 
As in the K\"ahler case, it is necessary for the existence of
a Sasaki-Einstein metric on $S$ 
that the transverse holomorphic
structure $(S,\Phi)$ has transversely positive first Chern class.
We would see that
there is further necessary condition, Proposition $5.9$ below,
for the existence of Sasaki-Einstein
metric.

\begin{definition}
A $p$-form $\alpha$ is called {\bf basic} if
$$i(\xi)\alpha=0,\ \ L_\xi\alpha=0,$$
where $\xi$ is the Reeb vector field on $S$, $i$ is the interior product
and $L_\xi$ is the Lie derivative with respect to $\xi$.
When we take into consideration the transversely holomorphic structure
$(\{U_a\}_{a\in A},\pi_a:U_a\to V_a)$ on $S$,
a $(p+q)$-form $\alpha$ is called a basic $(p,q)$-form if $\alpha$ is basic and
there is a $(p,q)$-form $\alpha_a$ on $V_a$ such that
$$\alpha_{|U_a}=\pi_a^*\alpha_a$$
for each $a\in A$.
Let $\Lambda_B^p$ (resp. $\Lambda_B^{p,q}$) be the sheaf of germs of
basic $p$-forms (resp. basic $(p,q)$-forms) and $\Omega_B^p=\Gamma
(S,\Lambda_B^p)$ (resp. $\Omega_B^{p,q}=
\Gamma(S,\Lambda_B^{p,q})$) the set of all global sections of
$\Lambda_B^p$ (resp. $\Lambda_B^{p,q}$).
\end{definition}

It is easy to see that $d\alpha$ is basic if $\alpha$ is basic.
We set $d_B=d_{|\Omega_B^*}$. Then $d_B^2=0$. Hence
we get a complex
$(\Omega_B^*,d_B)$
and call it the basic de Rham complex. We have the well-defined operators
$$\partial _B:\Omega_B^{p,q}\to \Omega_B^{p+1,q},\ \ 
\bar{\partial}_B:\Omega_B^{p,q}\to \Omega_B^{p,q+1}$$
which satisfy $d_B=\partial _B+\bar{\partial}_B$. The square of 
$\bar{\partial}_B$ vanishes and then we have a complex
$(\Omega_B^{p,*},\bar{\partial}_B)$, the basic Dolbeault complex.

\begin{example}
As we saw in the previous subsection, the transverse K\"ahler form
$\{\omega_\alpha^T\}_{\alpha\in A}$ of a Sasaki manifold $(S,g;\xi,\eta,\Phi)$
satisfies $\pi_\alpha^*\omega_\alpha^T=d\eta/2_{|U_\alpha}$.
Thus they are glued together and give a $d_B$-closed basic $(1,1)$-form
$d\eta/2$ on $S$.
We also call $\omega^T:=d\eta/2$ the transverse K\"ahler form.
Similarly we see that the Ricci forms of the transverse K\"ahler metric
$\{\rho_\alpha^T\}_{\alpha\in A}$,
$\rho_\alpha^T=-\sqrt{-1}\partial \bar{\partial}\log \det (g_\alpha^T)$, are
glued together and give a $d_B$-closed basic $(1,1)$-form
$\rho^T$ on $S$. $\rho^T$ is called the transverse Ricci form.
\end{example}

Of course, the transverse Ricci form $\rho^T$ depends on Sasaki metrics $g$.
Nevertheless its basic de Rham cohomology class is invariant
under some deformations of Sasaki structure defined as follows.

\begin{proposition}
Let $(S,g;\xi,\eta,\Phi)$ be a Sasaki manifold and $\varphi$ a basic 
function on $S$ such that $d\eta+2\sqrt{-1}\partial_B\bar{\partial}_B
\varphi$ is positive on the transverse of $\xi$.
Then we have a new Sasaki manifold $(S,g_\varphi;\xi,\eta_\varphi,\Phi)$,
where $\eta_\varphi=\eta+\sqrt{-1}(\bar{\partial}_B-\partial_B)\varphi$,
$\omega_\varphi^T=d\eta/2+\sqrt{-1}\partial_B\bar{\partial}_B\varphi$.
\end{proposition}

The basic de Rham class of the transverse K\"ahler form $\omega_\varphi^T$
is invariant under such deformation of Sasaki structures.
Similarly the basic de Rham class of the transverse Ricci form
$[\rho^T/2\pi]$ is an invariant of transverse holomorphic structure.
We call $[\rho^T/2\pi]$ the {\bf basic first Chern class} and denote by
$c_1^B(S,\Phi)$. A Sasaki manifold $(S,g;\xi,\eta,\Phi)$ is said to be
transversely positive if the basic first Chern class is represented by
a transversely positive $d_B$-closed $(1,1)$-form.
Sasaki-Einstein manifold is transversely positive by Proposition $5.4$.
However we should
note that the basic first Chern class of a transversely
positive Sasaki manifold is not always represented by
a transverse K\"ahler form. In fact we see the following proposition.

\begin{proposition}[\cite{bgo}]
The basic first Chern class is represented by $\tau d\eta$ for some
constant $\tau$ if and only if $c_1(D)=0$. Here $D=\Ker \ \eta$.
\end{proposition}

Recall that the integral invariant $f$ on a K\"ahler manifold, which is 
defined in Chapter $2$, is a character on the Lie algebra of holomorphic
Hamiltonian vector fields. Then we would like to define
transverse holomorphic Hamiltonian vector fields as follows.

\begin{definition}
We call a complex vector field $X$ on a Sasaki manifold $(S,g;\xi,\eta,\Phi)$
{\bf transverse holomorphic Hamiltonian} if it satisfies the following
two conditions:

$(1)$ $d\pi_\alpha(X)$ is a holomorphic vector field on $V_\alpha$. 

$(2)$ The complex valued function $u_X=\sqrt{-1}\eta(X)/2$ satisfies
$$\bar{\partial}_Bu_X=-\frac{\sqrt{-1}}{2}i(X)d\eta.$$
We denote by $\mathfrak{h}(S,\xi,\Phi)$ the set of all transverse holomorphic
Hamiltonian vector fields.
\end{definition}

Now we define an integral invariant of Sasaki manifold.
Let $(S,g';\xi',\eta',\Phi)$ be a transversely
$(2m+1)$-dimensional compact Sask. manifold.
Suppose that it is transversely positive and that $c_1(\Ker \ \eta')=0$.
Then if we choose a constant $a>0$ properly, the $D$-nomothetic transformed
Sask. manifold $(S,g;\xi,\eta,\Phi)$ satisfies
\begin{equation}
c_1^B(S,\Phi)=(2m+2)[d\eta/2]_B.
\end{equation}
Here 
$$\xi=\frac{1}{a}\xi',\ \ \eta=a\eta',\ \ g=erg'+(a^2-a)\eta'\otimes \eta',$$
$c_1^B(S,\Phi)$ is the basic first Chern class and
$[d\eta/2]_B$ is the basic K\"ahler class.
By a result of El Kacimi-Aloe \cite{elkacimi} there is a basic function $h$
such that
\begin{equation}
\rho^T-(m+1)d\eta=\sqrt{-1}\partial_B\bar{\partial}_Bh.
\end{equation}

We set

\begin{equation}
f_\xi(X)=\int_S Xh\ (d\eta/2)^m\wedge \eta,\ \ \ X\in \mathfrak{h}(S,\xi,\Phi).
\end{equation}
Then we can prove the following
using similar arguments as the proof of Theorem $2.7$.

\begin{theorem}[\cite{bgs}, \cite{FOW}]

The linear function $f_\xi$ on $\mathfrak{h}(S,\xi,\Phi)$ is 
invariant under the deformation of Sask. structure $(S,g_\varphi;
\xi,\eta_\varphi,\Phi)$
by basic function $\varphi$. 
In particular
$f_\xi$ is a Lie algebra homomorphism on
$\mathfrak{h}(S,\xi,\Phi)$.
Further 
if there exists a basic function $\varphi$ such that
$(S,g_\varphi;\xi,\eta_\varphi,\Phi)$ is Sasaki-Einstein, then
$f_\xi$ vanishes identically.
\end{theorem}

\subsection{Toric Sasaki manifolds}

First of all, we define toric Sasaki manifolds. Then
we see a relation between toric Sasaki manifolds and rational convex
polyhedral cones.

\begin{definition}

A Sasaki manifold $(S,g;\xi,\eta,\Phi)$ is said to be a
{\bf toric Sasaki manifold} if the Riemannian cone $(C(S),\bar{g},J)$ is a
toric K\"ahler manifold.
\end{definition}

Let $(S,g;\xi,\eta,\Phi)$ be a $(2m+1)$ dimensional toric Sasaki manifold.
Then, by definition, $(m+1)$ dimensional torus $T^{m+1}$ acts on 
$(C(S),\bar{g},J)$ effectively, holomorphically and isometrically.
In this case
the moment map $\mu:C(S)\to \mathfrak{t}^*$ is given by
$$\langle \mu(x),X\rangle =r^2\Tilde{\eta}(X^\#(x)),$$
where $\mathfrak{t}^*$ is the dual of the Lie algebra $\mathfrak{t}$ of
$T^{m+1}$, $X\in \mathfrak{t}$ and $X^\#(x)=\frac d{dt}|_{t=0}\exp (tX)x$.

\begin{definition}
Let
$\mathbb{Z}_\mathfrak{t}:=\Ker \{\exp:\mathfrak{t}\to T^{m+1}\}$ be the
integral lattice of $\mathfrak{t}$. A subset $\mathcal{C}\subset
\mathfrak{t}^*$
is a {\bf rational convex polyhedral cone} if there exist
$\lambda_j\in \mathbb{Z}_\mathfrak{t}$, $j=1,\dots,d$, such that
$$\mathcal{C}=\{y\in \mathfrak{t}^*\ |\ \langle \lambda_j,y\rangle \ge 0,\ 
j=1,\cdots,d\}.$$
We assume that the set $\{\lambda_j\}$ is minimal in that for any $j$,
$$\mathcal{C}\not=
\{y\in \mathfrak{t}^*\ |\ \langle \lambda_k,y\rangle \ge 0,\ 
k\not=j\}$$
and that each $\lambda_j$ is primitive, i.e. $\lambda_j$ is not of the
form $\lambda_j=a\mu$ for an integer $a\ge 2$ and $\mu \in \mathbb{Z}
_\mathfrak{t}$.
Under these two assumptions a rational convex polyhedral cone $\mathcal{C}$
with nonempty interior is {\bf good} if the following condition holds. If
$$\{y\in \mathcal{C}\ |\ \langle \lambda_{i_j},y\rangle\ge 0,\ j=1,\cdots,k\},
\ \ \ \{i_1,\cdots,i_k\}\subset \{1,\cdots,d\},$$
is a nonempty face of $\mathcal{C}$, then $\lambda_{i_1},\cdots,\lambda_{i_k}$
are linearly independent over $\mathbb{Z}$ and
\begin{equation}
\left\{\sum _{j=1}^k a_j\lambda_{i_j}\ |\ a_j\in \mathbb{R}\right\}\cap
\mathbb{Z}_\mathfrak{t}=\left\{\sum _{j=1}^k a_j\lambda_{i_j}\ |\ a_j\in 
\mathbb{Z}\right\}.
\end{equation}
\end{definition}

\begin{lemma}[\cite{lerman}]
Let $(S,g)$ be a $(2m+1)$ dimensional compact toric Sasaki manifold.
Then the image of the moment map $\mathcal{C}(\mu):=\mu(C(S))$
is a good rational convex polyhedral cone.
Further 
there is a vector $\alpha$ in the interior of the dual cone
$$\mathcal{C}(\mu)^*=\{\alpha\in \mathfrak{t}\ |\ \langle \alpha,
X\rangle \ge 0\ \text{for any }X\in \mathcal{C}(\mu)\}$$
such that the Reeb vector field $\Tilde{\xi}$ is generated by $\alpha$, that
is $\Tilde{\xi}=\alpha^\#$. 

Conversely, if a good rational convex polyhedral cone $\mathcal{C}\subset 
\mathfrak{t}^*$ and a vector $\alpha$ in the interior of
$\mathcal{C}^*\subset \mathfrak{t}$ are given, then we can construct, by 
Delzant construction, a $(2m+1)$ dimensional compact toric Sasaki manifold
such that the image of the moment map is $\mathcal{C}$ and the Reeb vector
field is $\alpha^\#$, see Proposition $3.4$ of \cite{CFO}.
Such a Sasaki manifold
is irregular if and only if $\alpha$ is an
irrational point,
that is $\alpha\notin \mathbb{Q}_\mathfrak{t}=\mathbb{Z}_\mathfrak{t}\otimes
\mathbb{Q}$.
\end{lemma}

We next see when a toric Sasaki manifold $(S,g;\xi,\eta,\Phi)$ is transversely
positive and satisfies $c_1(D)=0$. We identify $\mathfrak{t}^*\simeq
\mathbb{R}^{m+1}\simeq \mathfrak{t}$.

\begin{definition}[\cite{CFO}]
Let $\mathcal{C}=\{y\ |\ \langle \lambda_j,y\rangle \ge 0,\ 
j=1,\cdots,d\}\subset \mathbb{R}^{m+1}$ be a good rational convex polyhedral
cone. We call $\mathcal{C}$ a {\bf toric diagram of height }$l$ if
there exists $g\in SL(m+1,\mathbb{Z})$ such that
$g\lambda_j=(l,\lambda_j^1,\cdots,\lambda_j^m)$ for each $j$.
From now on, we always replace $\lambda_j$ by $g\lambda_j$
and assume that the $\lambda_j$ is of the form
$(l,\lambda_j^1,\cdots,\lambda_j^m)$.

\end{definition}

\begin{theorem}[\cite{CFO}]
Let $(S,g;\xi,\eta,\Phi)$ be a $(2m+1)$ dimensional compact toric Sasaki
manifold. If the equation $(14)$ holds, then there is a positive integer $l$
such that the image of the moment map $\mathcal{C}\subset \mathbb{R}^{m+1}$
is a toric diagram of height $l$. Further, 
the Reeb vector field can be written
as $\Tilde{\xi}=\alpha^\#$,
$\alpha\in \mathcal{C}_c^*:=
(\text{the interior of }\mathcal{C}^*)
\cap \{(y^0,\cdots,y^m)\in \mathbb{R}^{m+1}\ |\ y^0=l(m+1)\}$.

Conversely, if a toric diagram $\mathcal{C}\subset \mathbb{R}^{m+1}$ of 
height $l$
and $\alpha\in \mathcal{C}_c^*$ are
given, then the $(2m+1)$ dimensional toric Sasaki manifold given by
Delzant construction as in Lemma $5.14$ satisfies $(14)$.
In this case the $l$-th power $K_{C(S)}^{\otimes l}$
of the canonical bundle $K_{C(S)}$ of $C(S)$ is trivial.
\end{theorem}

\begin{example}
Let $m=2$. Then toric diagrams of height $1$ are obtained as 
follows\footnote{By \cite{lerman2}, if $S$ is simply connected, then
the image of the moment map is a toric diagram of height $1$.
However the converse is not always true, see \cite{CFO}.}.
Let $\Delta\subset \mathbb{R}^2$ be an integral convex polygon and
$v_j=(p_j,q_j)\in \mathbb{Z}^2,\ j=1,\cdots,d$ its vertices with 
counterclockwise order. Then
$\mathcal{C}_\Delta=\{(x,y,z)\ |\ x+p_jy+q_jz\ge 0,\ j=1,\cdots,d\}$
is a rational convex polyhedral cone in $\mathbb{R}^3$.
$\mathcal{C}_\Delta$ is a toric diagram of height $1$ if and only if
$\mathcal{C}_\Delta$ is good.

\begin{proposition}
$\mathcal{C}_\Delta$ is good if and only if either

$1$. $\lvert p_j-p_{j+1}\rvert=1$ or $\lvert q_j-q_{j+1}\rvert=1$

or

$2$. $\lvert p_j-p_{j+1}\rvert$ and $\lvert q_j-q_{j+1}\rvert$ are
relatively prime non-zero integers

for $j=1,\cdots,d$ where we have put $v_{d+1}=v_1$. 
\end{proposition}
\end{example}

\subsection{Sasaki-Ricci solitons}

To investigate the existence problem of Sasaki-Einstein metrics,
we introduce transverse K\"ahler-Ricci soliton (Sasaki-Ricci soliton).

\begin{definition}
A $(2m+1)$ dimensional Sasaki manifold $(S,g;\xi,\eta,\Phi)$ with a Hamiltonian
holomorphic vector field $X$ is called a {\bf K\"ahler-Ricci soliton}
or {\bf Sasaki-Ricci soliton} if
\begin{equation}
\rho^T-(2m+2)\omega^T=L_X\omega^T
\end{equation}
holds. Here $\rho^T$ and $\omega^T=d\eta/2$ are the transverse Ricci form and
the transverse K\"ahler form respectively.
\end{definition}

If $(S,g;\xi,\eta,\Phi)$ with $X$ is a Sasaki-Ricci soliton, then
$c_1^B(S,\Phi)=(2m+2)[\omega^T]_B$. Moreover
when $X=0$, $(S,g)$ is a Sasaki-Einstein manifold
by Proposition $5.4$.

We next consider the existence problem of Sasaki-Ricci solitons.
We need to consider ``normalized transverse
holomorphic Hamiltonian vector fields", whose corresponding Hamiltonian
function $u_X$ satisfying
\begin{equation}
\int_S u_Xe^h\omega^T\wedge \eta=0.
\end{equation}
Here $h$ is the real valued function on $S$ defined by $(15)$.
For any transverse holomorphic Hamiltonian vector field $X$, there is a
unique constant $c\in \mathbb{R}$ such that $X+c\xi$ is a normalized
transverse holomorphic Hamiltonian vector field. For simplicity of notation,
from now on any transverse holomorphic Hamiltonian vector field $X$ we consider
is normalized and its Hamiltonian function is denoted by $\theta_X$. 
Hence $\theta_X$ satisfies $(19)$.

As in Tian and Zhu \cite{TianZhu} we define a generalized
integral invariant $f_X$ for a given transverse holomorphic Hamiltonian 
vector field $X$ by
$$f_X(v)=-\int_S \theta_ve^{\theta_X}\omega^T\wedge \eta.$$
We can see that $f_X$ gives an invariant of the transverse holomorphic
structure $(S,\xi,\Phi)$. Moreover it gives an obstruction to the
existence of Sasaki-Ricci soliton, that is, if $(S,g;\xi,\eta,\Phi,X)$
is a Sasaki-Ricci soliton then $f_X(v)=0$ for any $v\in \mathfrak{h}
(S,\xi,\Phi)$. Note here that when $X=0$ the invariant $f_0$ is a
constant multiple of the integral invariant $f_\xi$ defined by $(16)$.

\begin{proposition}
Let $(S,g;\xi,\eta,\Phi)$ be a $(2m+1)$ dimensional compact Sasaki manifold
satisfying $(14)$. Then there exists a normalized transverse holomorphic
Hamiltonian vector field
$X\in \mathfrak{h}
(S,\xi,\Phi)$ such that $f_X\equiv 0$.
\end{proposition}

Let $(S,g;\xi,\eta,\Phi)$ be a $(2m+1)$ dimensional compact Sasaki manifold
which satisfies $(14)$ and $X\in \mathfrak{h}(S,\xi,\Phi)$ such that
$f_X\equiv 0$. Then the Sasaki structure $(S,g_\varphi;\xi,\eta_\varphi,\Phi)$
defined by a real valued basic function $\varphi$, see Lemma $5.8$,
is a Sasaki-Ricci soliton if and only if the following Monge-Amp\`ere
equation
\begin{equation}
\frac{\det(g_{i\bar{j}}^T+\varphi_{i\bar{j}})}{\det(g_{i\bar{j}}^T)}=
\exp (-(2m+2)\varphi-\theta_X-X\varphi+h),\ \ 
(g_{i\bar{j}}^T+\varphi_{i\bar{j}})>0
\end{equation}
holds, \cite {FOW}. Here $g_{i\bar{j}}^T$ is the components of the
transverse K\"ahler metric and $\varphi_{i\bar{j}}=\partial ^2\varphi/
\partial z^i\partial \overline{z^j}$, where $\{z^i\}$ are holomorphic
coordinates of
$V_\alpha$.
Using the continuity method, 
we see that there exists a Sasaki-Ricci soliton
if we have a priori $C^0$-estimate of $\varphi$.
When $(S,g;\xi,\eta,\Phi)$ is toric, we can verify a priori estimate
as Wang and Zhu \cite{WangZhu} showed in toric K\"ahler case.

\begin{theorem}[\cite{FOW}]
Let $(S,g;\xi,\eta,\Phi)$ be a $(2m+1)$ dimensional compact toric Sasaki
manifold which satisfies $(14)$ and $X\in \mathfrak{h}(S,\xi,\Phi)$
such that $f_X\equiv 0$. Then there exists a $T^{m+1}$-invariant
real valued basic function $\varphi$ such that $(S,g_\varphi
;\xi,\eta_\varphi,\Phi,X)$ is a Sasaki-Ricci soliton.
\end{theorem}

\begin{corollary}[\cite{FOW}]
Let $(S,g;\xi,\eta,\Phi)$ be a $(2m+1)$ dimensional compact toric Sasaki
manifold which satisfies $(14)$. If the integral invariant $f_\xi$
defined by $(16)$ identically vanishes, then there exists a $T^{m+1}$-invariant
real valued basic function $\varphi$ such that $(S,g_\varphi;\xi,\eta_\varphi,
\Phi)$ is a toric Sasaki-Einstein manifold.
\end{corollary}

\subsection{Volume minimization}

By Corollary $5.22$, in toric Sasaki case, the integral invariant $f_\xi$ 
vanishes identically
for $\xi=\alpha^\#,\ \alpha\in \mathcal{C}_c^*$ if and only if 
there is a toric Sasaki-Einstein metric whose Reeb vector field is $\xi$.
In the present section, we see that there exists always
such a Reeb vector field by virtue of the ``volume minimization property",
which was introduced by Martelli, Sparks and Yau in \cite{MSY1}, \cite{MSY2}.
Note here that Wang and Zhu \cite{WangZhu} proved that a toric Fano manifold
admits a K\"ahler-Einstein metric if and only if the integral invariant $f$
defined by $(4)$ identically vanishes. However it is well-known that
there exist many toric Fano manifolds such that $f\not\equiv 0$.
For example, $\mathbb{C}P^2\#k\overline{\mathbb{C}P^2},\ k=1,2$ does not
admit K\"ahler-Einstein metric.

Let $S$ be a $(2m+1)$ dimensional compact manifold and $Riem(S)$ the set
of all Riemannian metrics on $S$. If $g_0\in Riem(S)$ is an Einstein metric
with Einstein constant $2m$, then $g_0$ is a critical point of
the Einstein-Hilbert functional
$$\mathcal{S}(g):=\int_S(s(g)+2m(1-2m))dvol_g,$$
where $s(g)$ and $dvol_g$ are the scalar curvature and the volume element
of $g$ respectively.
(See, for example, Chapter $4$ of \cite{Besse}.)
Therefore, if there exists a Sasaki-Einstein metric on $S$, it is a
critical point of the Einstein-Hilbert functional.
``Volume minimizing property" of Sasaki-Einstein metrics follows from
this fact and Proposition $5.24$ below.

Now, we would like to define an appropriate deformation space of Sasaki metrics
from the K\"ahler cone viewpoint. We have dealt with deformations
of Sasaki metrics by basic functions (Proposition $5.8$)
so far. Such deformations fix the Reeb vector field and the transverse
holomorphic structure.
They are suitable to investigate the transverse K\"ahler geometry.
However, for volume minimization,
it is essential to consider
deformations of Sasaki metrics which change Reeb vector fields.

Let $(S,g_0;\xi_0,\eta_0,\Phi_0)$ be a $(2m+1)$ dimensional compact
Sasaki manifold. Suppose that the cone $(C(S),\bar{g}_0,J)$ is a K\"ahler
manifold with $c_1(C(S))=0$. For instance, when $c_1^B(S,\Phi_0)=
(2m+2)[d\eta_0/2]_B$, this condition holds.
We denote by $T$ the maximal torus of the holomorphic isometry group
of $(C(S),\bar{g}_0,J)$.

\begin{definition}
Let $\bar{g}$ be a K\"ahler metric on the complex manifold $(C(S),J)$.
We call $\bar{g}$ a {\bf K\"ahler cone metric} if there exist a
Riemannian metric $g$ on $S$ and a diffeomorphism
$\Psi_{\bar{g}}:C(S)\to \mathbb{R}_+\times S$
such that $\bar{g}=\Psi_{\bar{g}}^*(ds^2+s^2g)$, where
$s$ is the standard coordinate of $\mathbb{R}_+$.
Then $g$ is a Sasaki metric on $S$.
\end{definition}

The Reeb vector field and the contact form of $\bar{g}$,
viewed as the vector field and
the $1$-form on $C(M)$ respectively, are
$$\Tilde{\xi}_{\bar{g}}=Jr_{\bar{g}}\frac{\partial}{\partial r_{\bar{g}}},\ \ 
\Tilde{\eta}_{\bar{g}}=\sqrt{-1}(\bar{\partial}-\partial)\log r_{\bar{g}}.$$
Here $r_{\bar{g}}=pr_1\circ \Psi_{\bar{g}}$, $pr_1:\mathbb{R}_+\times S\to
\mathbb{R}_+$ is the projection.

Then we denote by $KCM(C(S),J)$ the set of all K\"ahler cone metrics
on $(C(S),J)$ such that the maximal torus of the holomorphic isometry group
is $T$. By identifying $\bar{g}$ and $g$, we can regard
$KCM(C(S),J)$ as a deformation space of Sasaki metrics on $S$.
As in the case of toric Sasaki manifolds,
the image $\mathcal{C}$
of the moment map with respect to $\bar{g}\in KCM(C(S),J)$
is a rational convex polyhedral cone in $\mathfrak{t}^*$,
the dual of the Lie algebra $\mathfrak{t}$ of $T$.
Moreover we see that the Reeb vector field is generated
by an element in the interior of the dual cone $\mathcal{C}^*\subset
\mathfrak{t}$
of $\mathcal{C}$.
Especially, if the condition $(14)$ holds, 
such an element
is included in the interior of a convex polytope $\mathcal{C}_c^*$,
which is the intersection of
$\mathcal{C}^*$ with an affine hyperplane in $\mathfrak{t}$.
See Section $2.6$ of \cite{MSY2} or Proposition $6.8$ of \cite{FOW}.
So we put
$$KCM_c=\{\bar{g}\in KCM(C(S),J)\ |\ \Tilde{\xi}_{\bar{g}}\in \mathcal{C}_c^*
\}$$
and $KCM_c(\Tilde{\xi})=\{\bar{g}\in KCM_c\ |\ \Tilde{\xi}_{\bar{g}}=
\Tilde{\xi}\}$
for each $\Tilde{\xi}\in \mathcal{C}_c^*$. 
By Proposition $5.4$, if $\bar{g}\in KCM(C(S),J)$ is Ricci-flat,
in other words, $g$ is Sasaki-Einstein, then $\bar{g}\in KCM_c$.

If we restrict the Einstein-Hilbert functional $\mathcal{S}$
to $KCM_c$, it is proportional to the volume functional:
\begin{proposition}[\cite{MSY2}]
Let
$\bar{g}\in KCM_c$.
Then we have
\begin{equation}
\mathcal{S}(g)=4m\Vol(S,g).
\end{equation}
\end{proposition}

We have the following first variation formula of the volume functional
on $KCM(C(S),J)$.

\begin{proposition}[\cite{MSY2}]
Let
$\{\overline{g_t}\}_{-\varepsilon<0<\varepsilon}$
be a $1$-parameter family of K\"ahler cone metrics in $KCM(C(S),J)$.
Then

\begin{equation}
\frac{d}{dt}\Vol(S,g_t)_{|t=0}=-(m+1)\int_S\eta(X)dvol
\end{equation}
holds. Here $\eta$ and $dvol$ are the contact form and the volume element
of $g_0$ respectively and $\mathfrak{t}\ni X=d\xi_t/dt_{|t=0}$.
\end{proposition}

\begin{corollary}
The volume functional on $KCM_c(\Tilde{\xi})$ is constant for each
$\Tilde{\xi}\in \mathcal{C}_c^*$. 
\end{corollary}

Hence, by Lemma $5.24$ and Corollary $5.26$,
$\mathcal{S}_{|KCM_c}$ is reduced to a function on $\mathcal{C}_c^*$
and we denote it by $\Tilde{\mathcal{S}}:\mathcal{C}_c^*
\to \mathbb{R}$. By the second variation formula of the volume function,
see \cite{MSY2} or \cite{FOW}, we get the following theorem.

\begin{theorem}[Volume minimization of Sasaki-Einstein metric, \cite{MSY2}]
If $\bar{g}\in KCM_c$ is Ricci-flat, then the Reeb vector field
$\Tilde{\xi}_{\bar{g}}\in \mathcal{C}_c^*$ is the unique minimum point
of $\Tilde{\mathcal{S}}$.
\end{theorem}

\begin{example}
Let 
$(S,g_0;\xi_0,\eta_0,\Phi_0)$
be a $(2m+1)$-dimensional compact toric Sasaki manifold
satisfying the condition $(14)$.
Then $KCM_c$ consists of toric K\"aher cone metrics on $(C(S),J)$
such that the corresponding Sasaki metric satisfies $(14)$.
In this case, as we saw in Section $5.3$, the image $\mathcal{C}$
of the moment map is a toric diagram of height $l$ in
$\mathbb{R}^{m+1}$. We also see that
$\Tilde{\mathcal{S}}$ is 
$$\Tilde{\mathcal{S}}(\Tilde{\xi})=8m(m+1)(2\pi)^{m+1}
\Vol(\Delta(\Tilde{\xi})),$$
see \cite{MSY1}. Here 
$\Delta(\Tilde{\xi})=\{x\in \mathcal{C}\ |\ \Tilde{\xi}\cdot x\le 1\}$
and $\Vol(\Delta(\Tilde{\xi}))$ is the Euclidean volume of 
$\Delta(\Tilde{\xi})$.
This function is convex and proper on the interior of the $m$ dimensional
convex polytope $\mathcal{C}_c^*$.
Therefore $\Tilde{\mathcal{S}}$ has the unique minimizer $\Tilde{\xi}_{min}\in 
\mathcal{C}_c^*$.
\end{example}

Now, Theorem $5.27$ gives a necessary condition for existence of a
Ricci-flat K\"ahler cone metric in $KCM_c(\Tilde{\xi})$.
Actually, it is reduced to vanishing of the integral invariant
$f_{\Tilde{\xi}}$ by the following theorem.

\begin{theorem}[\cite{MSY2}, \cite{FOW}]
The first variation $d_{\Tilde{\xi}}\Tilde{\mathcal{S}}$ of
$\Tilde{\mathcal{S}}$
at $\Tilde{\xi}\in \mathcal{C}_c^*$ equals to
$-\sqrt{-1}f_{\Tilde{\xi}}$.
\end{theorem}

Hence we see the following existence results of toric Sasaki-Einstein
metrics, by Corollary $5.22$ and Theorem $5.29$.

\begin{theorem}[\cite{FOW}]
Let $(S,g;\xi,\eta,\Phi)$ be a $(2m+1)$ dimensional compact toric
Sasaki manifold satisfying $(14)$ and $\Tilde{\xi}_{min}\in \mathcal{C}_c^*$
the minimizer of $\Tilde{\mathcal{S}}$.
Then there exists a Ricci-flat K\"ahler cone metric $\overline{g_{SE}}$
in $KCM_c(\Tilde{\xi}_{min})$. The corresponding Riemannian metric
$g_{SE}$ is a toric Sasaki-Einstein metric on $S$.
On the other hand, there is no Ricci-flat K\"ahler cone metric
in $KCM_c(\Tilde{\xi})$ if $\Tilde{\xi}\not=\Tilde{\xi}_{min}$.
That is to say, there exists no Sasaki-Einstein metric with the Reeb vector
field $\Tilde{\xi}_{\bar{g}}\not=\Tilde{\xi}_{min}$.
\end{theorem}

\begin{example}
For the integral vectors
$$\lambda_1=(1,0,0),\ \lambda_2=(1,1,0),\ \lambda_3=(1,2,1),\ 
\lambda_4=(1,1,2),\ \lambda_5=(1,0,1),$$
the rational convex polyhedral cone
$\{x\in \mathbb{R}^3\ |\ x\cdot \lambda_i\ge 0,\ i=1,2,3,4,5\}$
in $\mathbb{R}^3$ is a $3$-dimensional toric diagram of height $1$, by
Lemma $5.18$.
Then the toric variety minus the apex $(C(S),J)$ given by Delzant construction
is $K_{\mathbb{C}P^2 \# 2\overline{\mathbb{C}P^2}}\setminus 
(\text{zero section})$, where $K_{\mathbb{C}P^2 \# 2\overline{\mathbb{C}P^2}}$
is the canonical bundle of the two-point blow-up of the
complex projective plane
$\mathbb{C}P^2 \# 2\overline{\mathbb{C}P^2}$.
By the computation of \cite{MSY1},
$$\Tilde{\xi}_{min}=(3,\frac{9}{16}(-1+\sqrt{33}),\frac{9}{16}(-1+\sqrt{33})).
$$

Therefore, in this case, there is an irregular toric Sasaki-Einstein
metric on the associated $S^1$-bundle $S(K_{\mathbb{C}P^2 \# 2
\overline{\mathbb{C}P^2}})$.
Note here that 
$\mathbb{C}P^2 \# 2\overline{\mathbb{C}P^2}$ does not admit
K\"ahler-Einstein metric, since the regular Reeb vector
is $(3,3,3)\not=\Tilde{\xi}_{min}$.
\end{example}


When $m=2$, we can get the following result
as an application of
Theorem $5.30$.

\begin{theorem}[\cite{CFO}]
For each positive integer $k$ there exists an infinite family of
inequivalent toric Sasaki-Einstein metrics on the $k$-fold connected
sum $\#k(S^2\times S^3)$ of $S^2\times S^3$.
\end{theorem}

The existence of Sasaki-Einstein metrics, which is possibly non-toric, on 
$\#k(S^2\times S^3)$ has been known by the works of Boyer, Galicki, Nakamaye
and Koll\'ar (\cite{BoGaKo05}, \cite{BoGaNa02}, \cite{Kol04}), and
that the existence of toric Sasaki-Einstein metrics for all odd $k$'s
has been known by van Coevering (\cite{Coev06}).
Hence our results is new in that we obtain toric constructions for all $k$'s.
Moreover most of our examples should be irregular while the previous ones
are all quasi-regular.

As another application of Theorem $5.30$, we
get the following.

\begin{theorem}[\cite{futaki07.2}]\label{main2} 
Let $M$ be a toric Fano manifold and $L$ a holomorphic line bundle
on $M$ such that $K_M = L^{\otimes p}$ for some positive integer $p$.
Then, for each positive integer $k$, there exists a
complete scalar-flat K\"ahler metric on the total space of $L^{\otimes k}$.
In particular, when $k=p$, it is Ricci-flat.
\end{theorem}

When $M$ is the one-point blow-up of the complex projective plane,
Oota-Yasui \cite{OY06} constructed such complete 
Ricci-flat K\"ahler metric explicitly, but their metric is different from
the one constructed in \cite{futaki07.2}.

\begin{theorem}[\cite{futaki07.2}]\label{main3} 
Let $(S,g)$ be a compact Sasaki-Einstein manifold.
We denote by $(C(S),J,\bar{g})$ the K\"ahler cone manifold.
Then the following statements hold.
\begin{enumerate}
\item[(a)]  There exists a complete scalar-flat K\"ahler metric on
$(C(S),J)$．
\item[(b)] For any negative constant $c$, there is $\gamma>0$ such that
there exists a complete K\"ahler metric of scalar curvature $c$
on the submanifold $\{0 < r < \gamma \}\subset C(S)$.
\end{enumerate}
Hence a toric K\"ahler cone $C(S)$ obtained from a toric diagram
admits a complete scalar-flat K\"ahler metric.
\end{theorem}

As a special case of Theorems $5.33$ and $5.34$, we have the following.

\begin{theorem}[\cite{futaki07.2}]\label{TSE3}
Let $M$ be a toric Fano manifold. Then
there exists a complete Ricci-flat K\"ahler metric on the total space
of the canonical bundle $K_M$.
\end{theorem}

Theorem \ref{TSE3} is an extension of the Eguchi-Hanson metric
($M = \bfC\bfP^1$, \cite{eguchi-hanson}) or
the Calabi metric ($M = \bfC\bfP^m$, \cite{calabi79}).

\begin{theorem}[\cite{futaki07.2}]\label{TSE4}
Let $M$ be a toric Fano manifold. Then
there exists a complete scalar-flat K\"ahler metric on $K_M
\setminus\{\text{zero section}\}$.
\end{theorem}

We can prove Theorems \ref{TSE3} and \ref{TSE4} by applying
the moment construction (\cite{Hwang-Singer02}) to $\eta$-Einstein Sasaki
manifolds.

\subsection{Uniqueness of toric Sasaki-Einstein metrics}

In the present section, we see the uniqueness of Sasaki-Einstein metrics
on compact toric Sasaki manifolds modulo the action of the
identity component of the automorphism group for the transverse holomorphic
structure.

\begin{definition}
Let $(S,g;\xi,\eta,\Phi)$ be a Sasaki manifold and
$(C(S),\bar{g},J)$ its K\"ahler cone.
We call an automorphism  of $(C(S),J)$
an {\bf automorphism of transverse holomorphic structure}
if it commutes with the holomorphic flow generated by
$\Tilde{\xi}-\sqrt{-1}J\Tilde{\xi}$.
We denote by $Aut(C(S),\Tilde{\xi})_0$
the identity component of the group of the automorphism of 
the transverse holomorphic structure.
\end{definition}

In K\"ahler geometry a well-known method of proving uniqueness of constant
scalar curvature metrics is to use geodesics on the space of all K\"ahler
metrics in a fixed K\"ahler class, see \cite{Chen}, \cite{Mabuchi}
for example. This idea becomes substantially simpler when the K\"ahler
manifold under consideration is toric because the geodesic becomes
a line segment expressed by the symplectic potentials, which is
the Legendre dual of the K\"ahler potentials, see \cite{Guan}.
In the Sasaki case, we can prove the uniqueness of toric Sasaki-Einstein
metric using a similar idea.

Let $(S,g;\xi,\eta,\Phi)$ be a $(2m+1)$-dimensional compact Sasaki
manifold satisfying the condition $(14)$ and put
$$\mathcal{K}(\Tilde{\xi}):=\{\varphi:T\text{-invariant basic function}
\ |\ \omega^T+\sqrt{-1}\partial_B\bar{\partial}_B\varphi>0\},$$
where 
$\mathcal{K}(\Tilde{\xi})/\mathbb{R}\simeq KCM_c(\Tilde{\xi}),\ 
\varphi+$(constant)$\mapsto g_\varphi$ ($g_\varphi$ is the Sasaki metric
obtained as Proposition $5.8$).
Then we define the equation of geodesics $\{\varphi_t\}$ in
$\mathcal{K}(\Tilde{\xi})$ as
\begin{equation}
\ddot{\varphi}_t-\lvert \bar{\partial}\dot{\varphi}_t\rvert _{\omega_t}^2=0
\end{equation}

As in K\"ahler case, we can show that the existence of geodesics
induces the uniqueness of Sasaki-Einstein metrics,
since the ``transverse Mabuchi energy" is convex.

\begin{proposition}
Let $(S,g_{\varphi_i};\xi,\eta_{\varphi_i},\Phi),\ i=1,2$, be
Sasaki-Einstein manifolds. If there exists a geodesic in
$\mathcal{K}(\Tilde{\xi})$ connecting $\varphi_1$ and $\varphi_2$,
then $\alpha^*g_{\varphi_2}=g_{\varphi_1}$ for some 
$\alpha\in Aut(C(S),\Tilde{\xi})_0$. 
\end{proposition}

Especially, when $(S,g;\xi,\eta,\Phi)$ is a toric Sasaki manifold,
we see the existence of geodesics, as Guan's procedure in the
toric K\"ahler case (\cite{Guan}).

\begin{theorem}[\cite{CFO}]
Let $(S,g_{\varphi_i};\xi_{min},\eta_{\varphi_i},\Phi),\ i=1,2$, be
compact toric Sasaki-Einstein manifolds.
Then $\alpha^*g_{\varphi_2}=
g_{\varphi_1}$ for some $\alpha\in Aut(C(S),\Tilde{\xi})_0$.
\end{theorem}

\subsection{Obstructions to the existence of Sasaki-Einstein metrics}

By Theorems $5.30$ and $5.39$, we have solved the existence and 
the uniqueness of Sasaki-Einstein metrics in the compact toric case.
Then, does there exist a Sasaki-Einstein metric in the non-toric case?
The answer to this question, in general, is no. There are some obstructions
to the existence of Sasaki-Einstein metrics. They were suggested by
Gauntlett, Martelli, Sparks and Yau in \cite{gmsy}.
In the present section, we see such obstructions, called the Bishop obstruction
and the Lichnerowicz obstruction.

Let $(S,g;\xi,\eta,\Phi)$ be a compact Sasaki manifold and
$(C(S),\bar{g},J)$ the K\"ahler cone.
To introduce the obstructions,
we must define an invariant of the triple $(C(S),J,
\Tilde{\xi})$, where $\Tilde{\xi}$ is the Reeb vector field of $g$.
We denote by $H(S)$ the $L^2$-closure of the set of all smooth functions $f$
on $S$ which can extend to holomorphic functions $\Tilde{f}$ on
$\{r\le 1\}\subset (C(S),J)$ with $\Tilde{f}\to 0\ (r\to 0)$.
$H(S)$ 
is called the Hardy space.
Then the operator $T=\xi _{|H(S)}/\sqrt{-1}$ on $H(S)$
is a first-order self-adjoint Toeplitz
operator with positive symbol.

\begin{proposition}
$T$ has non-negative discrete spectra. 
\end{proposition}
\begin{proof}
By \cite{BG}, $T$ has discrete spectra bounded from below.
Suppose that $f\in H(S)$ satisfies $Tf(=\xi f/\sqrt{-1})
=\lambda f$. Then we see that the holomorphic extension
$\Tilde{f}$ of $f$ is given by $\Tilde{f}=r^\lambda f$.
Hence $\lambda\ge 0$ by the definition of $H(S)$.
\end{proof}

Of course,
the eigenvalues of $T$ depends only on the triple $(C(S),J,\Tilde{\xi})$.
So they define invariants of $(C(S),J,\Tilde{\xi})$ called
{\bf charges}, \cite{MSY2}, \cite{gmsy}.

\begin{example}
Let $(S,g;\xi,\eta,\Phi)$ be a regular Sasaki manifold.
Then the Reeb vector field $\xi$ generates a free $S^1$ action on $S$ and
the Sasaki structure induces the K\"ahler structure on the quotient
space $M=S/S^1$. Moreover there is an ample line bundle $L$ over
$M$ such that
$S=S(L)$, where $S(L)$ is the total space of the associated $S^1$-bundle,
see Example $5.3$.
In such case, we see that
$$H(S)\simeq \bigoplus_{k=0}^\infty H^0(M;L^k)$$
and $H^0(M;L^k)$ is the charge $k$ eigenspace
for each non-negative integer $k$.
\end{example}

By Corollary $5.26$, the volume $\Vol(S,g)$
of a Sasaki manifold $(S,\xi,\Phi)$
is an invariant of $(C(S),J,\Tilde{\xi})$.
We can obtain the invariant $\Vol(S,g)$
from the asymptotic behavior of the charges.

\begin{theorem}[\cite{MSY2}, \cite{BG}]
Let $0=\lambda_0< \lambda_1\le\cdots$ be the charges of 
a compact $(2m+1)$-dimensional Sasaki manifold $(S,g;\xi,\eta,\Phi)$.
Then
\begin{equation}
\Vol(S,g)=\gamma_{2m+1}\lim_{t\searrow 0}t^{m+1}\sum_{j=0}^\infty \exp
(-t\lambda_j),
\end{equation}
where $\gamma_{2m+1}$ is the volume of the $(2m+1)$-dimensional unit sphere.
\end{theorem}

In addition the charges relate with the eigenvalues of the Laplacian of 
$(S,g)$.

\begin{proposition}
Let $f\in H(S)$. If
$Tf=\lambda f$, then $\Delta_Sf=\lambda(\lambda+2m)f$.
Here $\Delta_S$ is the Laplacian of $(S,g)$
acting on $C^\infty(S)$. 
\end{proposition}

\begin{proof}
The holomorphic extension $\Tilde{f}=r^\lambda f$ of $f$ is
harmonic, since $(\{r\le 1\},\bar{g},J)$ is K\"ahler.
Therefore
\begin{align*}
0 &= \Delta_{C(S)}\Tilde{f} =\frac{1}{r^2}\Delta_S (r^\lambda f)
-\frac{1}{r^{2m+1}}
\frac{\partial}{\partial r}\left(r^{2m+1}\frac{\partial}{\partial r}\right)
(r^\lambda f)\\
&= r^{\lambda-2}(\Delta_Sf-\lambda(\lambda+2m)f),
\end{align*}
where $\Delta_{C(S)}$ is the Laplacian of $(\{r\le 1\},\bar{g})$
acting on $C^\infty(\{r\le 1\})$.
\end{proof}

Now, we recall the following two theorems from Riemannain geometry.
The first one is a
theorem of Lichnerowicz \cite{Lic};
if $(S,g)$ is a $m$-dimensional complete Riemannian manifold
with $\text{Ric}\ge (m-1)g$, then $M$ is compact and the first positive
eigenvalue of the Laplacian is greater than or equal to $m$.
The second one is a theorem 
of Bishop \cite{BC};
if $(S,g)$ is a $m$-dimensional complete Riemannian manifold
with $\text{Ric}\ge (m-1)g$, 
then the volume $\Vol(S,g)$ is less than or equal to
$\gamma_m$,
the volume of the $m$-dimensional unit sphere.
As a result of these theorems, we have the following necessary conditions
for the existence of Sasaki-Einstein metrics.

\begin{theorem}[Lichnerowicz's obstruction, \cite{gmsy}]
Let $(S,g;\xi,\eta,\Phi)$ be a $(2m+1)$-dimensional compact Sasaki-Einstein
manifold. Then the first positive charge $\lambda_1$ is greater than
or equal to $1$.
\footnote{In \cite{gmsy}, it is indicated that
the first positive charge $\lambda_1$ of any compact regular Sasaki manifold
satisfying $(14)$ is greater than or equal to $1$.
Hence, this condition does not give new obstruction to the existence
of K\"ahler-Einstein metric.}. 
\end{theorem}

\begin{theorem}[Bishop's obstruction, \cite{gmsy}]
Let $(S,g;\xi,\eta,\Phi)$ be a $(2m+1)$-dimensional compact Sasaki-Einstein
manifold. Then the volume $\Vol(S,g)$, which is an invariant of 
$(S,\xi,\Phi)$ by Corollary $5.26$ or Theorem $5.42$,
is less than or equal to $\gamma_{2m+1}$.

\end{theorem}

\begin{example}[\cite{gmsy}, \cite{BG}]
We consider the action of $\mathbb{C}^*$ on $\mathbb{C}^{m+2}$
defined as
$$(z_0,\cdots,z_{m+1})\mapsto (q^{w_0}z_0,\cdots,q^{w_{m+1}}z_{m+1}),\ \ 
\boldsymbol{w}=(w_0,\cdots,w_{m+1})\in \mathbb{N}^{m+2},\ q\in \mathbb{C}^*.$$
Suppose that a polynomial $F$ on $\mathbb{C}^{m+2}$
satisfies
$$F(q^{w_0}z_0,\cdots,q^{w_{m+1}}z_{m+1})=q^dF(z_0,\cdots,z_{m+1}),\ \ 
d\in \mathbb{N}$$
and that $X=\{F=0\}\subset \mathbb{C}^{m+2}$ has no singular point
except the origin. 
Moreover we assume that $\lvert \boldsymbol{w}\rvert =\sum w_j>d$.
Note that this last condition corresponds to the Fano property of
the quotient $X/\mathbb{C}^*$.

Let $\zeta$ be the generator of the $S^1\subset \mathbb{C}^*$ action on
$X$. If we normalize it as
$$\Tilde{\xi}=\frac{m+1}{\lvert \boldsymbol{w}\rvert -d}\zeta,$$
then we see that $\Tilde{\xi}\in \mathcal{C}_c^*$, see \cite{gmsy}.
Of course, when $\Tilde{\xi}$ is not the minimizer of $\Tilde{S}$,
there exists no Ricci-flat metric in $KCM_c(\Tilde{\xi})$.
Thus suppose here that $\Tilde{\xi}$ is the minimizer of $\Tilde{S}$.
Then it is easy to see that
\begin{equation}
\lambda_1=\frac{(m+1)\min\{w_j\}}{\lvert \boldsymbol{w}\rvert -d},\ \ 
\Vol=\frac{d\gamma_{2m+1}
(\lvert \boldsymbol{w}\rvert -d)^{m+1}}{(m+1)^{m+1}\prod w_j}.
\end{equation}
If we choose $\boldsymbol{w}$ such that $\lambda_1$ or $\Vol$ in $(25)$
do not fulfil the conditions in Theorems $5.44$ or $5.45$,
then $X$ does not admit Ricci-flat K\"ahler cone metric whose Reeb vector field
is $\Tilde{\xi}$.

For example, let $F$ be a polynomial given by
$F(z_0,\cdots,z_{m+1})=z_0^{a_0}+\cdots+z_{m+1}^{a_{m+1}}$, where
$(a_0,a_1,\dots,a_{m+1})\in (\mathbb{Z}_{\ge 0})^{m+2}$.
Then the three conditions, $\lvert \boldsymbol{w}\rvert >d,
\lambda_1\ge 1$ and $\Vol\le \gamma_{2m+1}$ can be expressed in
$(a_0,a_1,\dots,a_{m+1})$ as follows.
$$\lvert \boldsymbol{w}\rvert >d\iff \frac{1}{a_0}+\cdots+\frac{1}{a_{m+1}}
>1$$
\begin{equation}
\lambda_1\ge 1\iff (m+1)\min\{1/a_j\}\ge \frac{1}{a_0}+
\cdots+\frac{1}{a_{m+1}}-1
\end{equation}

\begin{equation}
\Vol\le \gamma_{2m+1}\iff (\prod a_j)(\frac{1}{a_0}+\cdots+\frac{1}{a_{m+1}}
-1)\le (m+1)^{m+1}
\end{equation}
In case when $m=2$ and $a_0=a_1=a_2=2,\ a_3=k>4$, $(26)$ does not hold.
In this case, $\Tilde{\xi}$ is the minimizer of $\Tilde{S}$,
see \cite{gmsy}. Therefore $\{z_0^2+z_1^2+z_2^2+z_3^k=0\}\subset \mathbb{C}
^4$, $k>4$ admits no Ricci-flat K\"ahler cone metric
\footnote{Actually it is proved that there does not exist
a Ricci-flat K\"ahler cone metric on
$\{z_0^2+z_1^2+z_2^2+z_3^3=0\}$, see
\cite{gmsy}, 
\cite{conti}. }. 
\end{example}

Lastly, we would like to
comment on a relation between the existence of Sasaki-Einstein
metric and GIT-stability. 
It is hard to treat an irregular Sasaki manifold in the methods of
algebraic geometry. 
Therefore we feel that, in the Sasaki case,
there is no direct relation with GIT-stability nor K-stability
when $2$-dimensional torus acts isometrically.
However we can define the Bergman kernel (Szeg\"o kernel)
on the cone of a Sasaki manifold and can analyze them.
In fact we saw that the asymptotical behavior of charges has an important
information on the existence of Sasaki-Einstein metric.
Thus we could imagine that the existence of Sasaki-Einstein (or
constant scalar curvature Sasaki) metric is equivalent to some
asymptotic analytical conditions, not algebraic ones.

\end{document}